\RequirePackage{amsmath}
\documentclass{svjour3}                     % onecolumn (standard format)
\smartqed  % flush right qed marks, e.g. at end of proof
\usepackage{graphicx}
\usepackage{mathptmx}      % use Times fonts if available on your TeX system
\usepackage{bm}
\usepackage{amsfonts,amscd}

\usepackage{amsthm}
\usepackage{float,pstricks}
\usepackage{tikz-cd}
\newcommand{\C}{\mathbb{C}}

\newcommand{\vf}{\frac{d}{dz}}
\newcommand{\HH}{\mathbb{H}}
\newcommand{\R}{\mathbb{R}}
\newcommand{\Res}{\textrm{Res}}
\newcommand{\CC}{\mathcal{C}}

\newcommand{\eqpt}{equilibrium point}
\newcommand{\multeq}{multiple equilibrium point}
\newcommand{\codim}{{\rm codim}}
\newtheoremstyle{dotless}{}{}{\itshape}{}{\bfseries}{}{ }{}  %%% so dot does not appear after theorem number
\theoremstyle{dotless}
\newtheorem*{thmrank1comp}{Theorem \ref{rank1comp}}
\newtheorem*{thmrank1char}{Theorem \ref{rank1characterization}}
 \journalname{}
\begin{document}

\title{A Characterization of Multiplicity-Preserving Global Bifurcations of Complex Polynomial Vector Fields \thanks{This research was supported by  Fondation Idella, the Marie Curie European Union Research Training Network {\it Conformal Structures and Dynamics} (CODY), the Research Foundation of CUNY PSC-CUNY Cycle 44 (66148-00 44) and Cycle 47 (69510-00 47) Research Awards, the Bronx Community College Foundation Faculty Scholarship Grant 2016, and the Association for Women in Mathematics  Travel Grant  October 2019 Cycle (NSF 1642548). }
}
%\subtitle{Do you have a subtitle?\\ If so, write it here}

\titlerunning{Multiplicity-preserving Bifurcations of Vector Fields in $\C$}        % if too long for running head

\author{Kealey Dias}

%\authorrunning{Short form of author list} % if too long for running head

\institute{K. Dias \at
              Department of Mathematics and Computer Science \\
              Bronx Community College of the City University of New York \\
              2155 University Avenue \\
              Bronx, NY 10453 USA
              Tel.: +1-718-289-5030\\
              %Fax: +123-45-678910\\
              \email{kealey.dias@bcc.cuny.edu}            \\
              ORCID iD: 0000-0002-4613-6881
%             \emph{Present address:} of F. Author  %  if needed
           %\and
           %S. Author \at
              %second address
}

\date{Received: date / Accepted: date}
% The correct dates will be entered by the editor

\maketitle

\begin{abstract}
For the space of single-variable monic and centered complex polynomial vector fields of arbitrary degree $d$, it is proved that any bifurcation which preserves the multiplicity of equilibrium points can be realized as a composition of a finite number of simpler bifurcations, and  these bifurcations are characterized.
%Insert your abstract here. Include keywords, PACS and mathematical
%subject classification numbers as needed.
% \PACS{PACS code1 \and PACS code2 \and more}
 \subclass{MSC 37C10 %Vector fields, flows, ODEs
     \and MSC 34C23 %ODE-> Qualitative theory-> bifurcation
     \and MSC 34M99 %DEs in the complex domain, in this section but no subtopic
     \and MSC 37C29 %Dynamical systems -> smooth dyn sys -> Homoclinic and heteroclinic orbits
     %\and MSC 37E35 %dyn sys -> low dim dynamics -> flows on surfaces
     %\and MSC 37F30 %Dyn sys -> complex dyn -> Quasiconformal Methods and Teichmuller
     \and MSC 37F75 %holomorphic foliations and vector fields
     %\and MSC 30F30 %Differentials on Riemann Surfaces
     %\and MSC 34A26 %Geometric Methods in (ordinary) Differential Equations
     }
\end{abstract}
\section{Introduction}
\label{intro}
%\the\textwidth
Bifurcations, the qualitative change in dynamics produced by varying parameters, are fundamental to the analysis of any family of dynamical systems but are notoriously difficult to describe in any generality.  This paper makes a significant step towards a complete description of the bifurcations of the global topological structure of the integral curves of the complex polynomial vector fields
\begin{equation}
  \xi_P=P(z)\vf, \quad z\in \C,
\end{equation}
or equivalently, the maximal solutions $\gamma (t,z)$ to the associated autonomous ordinary differential equation (ODE)
\begin{equation}\label{objects}
  \dot{z}=P(z), \quad \gamma(0,z)=z, \quad z \in \mathbb{C}, \ t \in \mathbb{R},
\end{equation}
where $P(z)=z^d+a_{d-2}z^{d-2}+\dots+a_0$ is a monic and centered polynomial of degree $d\geq 2$. Namely, we characterize the \emph{multiplicity-preserving} bifurcations, i.e. bifurcations where the multiplicities of the equilibrium points (the zeros of $P$) are preserved under small perturbation. For significant work on parabolic bifurcations (which do \emph{not} preserve multiplicities) of complex vector fields, see \cite{MRR2004}, \cite{CR2014}, and \cite{Rousseau2015}.\par
Complex polynomial ODEs of the form \eqref{objects} are a subset of the $\mathbb{R}^2$ systems
\begin{align}\label{realode}
  \dot{x}&=u(x,y)\nonumber \\
  \dot{y}&=v(x,y), \quad x,\ y, \ t \in \mathbb{R},
\end{align}
whose global qualitative structure in general remains a fundamental open problem in dynamics.
 Famously, part of  Hilbert's 16th problem inquires to the number and configurations of limit cycles in the plane for each degree $d$ polynomial system in two real variables.
Even though holomorphic vector fields, which lack limit cycles (e.g. \cite{Sve1978,Nee1994}), may seem distant from Hilbert's 16th problem, experts in this area have shown that a significant class of perturbations to study  are non-holomorphic perturbations of holomorphic polynomial vector fields with  centers \cite{AGP2010}, \cite{LS04}, \cite{ALS2005}.   Creating a complete description of bifurcations for complex polynomial vector fields would not only answer a fundamental question about holomorphic systems in their own right, but understanding the holomorphic perturbations may reduce the complexity of the analysis of  non-holomorphic perturbations.   \par
  Complex vector fields are also linked to other areas of mathematics. In particular, the properties of single-variable complex vector fields have been utilized in proving prominent results  in iterated complex dynamics, including the study of parabolic bifurcations (e.g.  \cite{Shi}, \cite{Ben1993}, \cite{Outhesis}, \cite{BT2007}) and in the proof that there exist quadratic polynomial Julia sets of positive Lebesgue measure \cite{BC2005}. Single-variable complex vector fields are  being used to study higher dimensional complex systems \cite{RT2008}, as well as  quadratic and Abelian differentials on Riemann surfaces. The integral curves of holomorphic vector fields on Riemann surfaces foliate  the Riemann surface, and they are similar in structure to the trajectories  of these differentials.  Understanding the bifurcations of complex ODEs may contribute to one of these fields in a way that has not yet been explored.\par
%%%%%%-----------------------------------------
Even when restricting to the study of holomorphic or meromorphic planar systems in their own right, a comprehensive study of the bifurcations remains elusive.  Researchers have made headway by considering lower-degree polynomial systems so that the bifurcation diagram can be visualized \cite{GH,LS04,RT2008}, or describing bifurcations induced by rotations of vector fields \cite{JMRCVV95}.
The aim of this paper is to make to make a significant step towards a comprehensive analysis of all global bifurcations for polynomial vector fields in $\mathbb{C}$ of arbitrary degree by characterizing the bifurcations where the multiplicity of the equilibrium points is preserved.\par
%%%%%%-----------------------------------------
This paper is organized as follows. Sections \ref{prelimssection} and \ref{parmspacesection} review basic properties of complex ODEs and introduce parameter space and bifurcations for the systems under consideration.
Section \ref{deformationssection} describes deformations in rectifying coordinates and proves that these are enough to study the multiplicity-preserving bifurcations.
Rank $k$ bifurcations are defined in Section \ref{rank1compsection}, and the first main theorem of this paper is proved:
\begin{thmrank1comp}
Every multiplicity-preserving bifurcation can be realized as a composition of rank 1 bifurcations.
\end{thmrank1comp}
Section \ref{characterizerank1} characterizes the rank 1 bifurcations to show that they are all essentially the same type. The informal statement of the second main theorem in this paper is:
\begin{thmrank1char}[informally]
Every rank 1, multiplicity-preserving  bifurcation is of the type where a sequence of $n\geq 1$ homoclinic separatrices (satisfying some technical conditions) break, and $n-1$ (specified) homoclinic separatrices form under  perturbation.
\end{thmrank1char}
%Finally, Section \ref{futuresection} outlines natural extensions of this work.
%%%%%%--------------------------------------------
\section{Preliminaries}
\label{prelimssection}
The complex ODEs \eqref{objects} inherit the general properties of the real planar ODEs \eqref{realode},  including existence and uniqueness of solutions,  dependence on initial conditions, Hartman-Grobman  theorem (linearization), and long-term behavior described by the Poincare-Bendixon theorem. \par
\paragraph{Equilibrium points} Classification of the local dynamics of holomorphic vector fields near an equilibrium point, a zero $\zeta$ of $P$, is  known to depend only on the order of the zero and the  \emph{dynamical residue} $\Res(1/P,\zeta)$ \cite{JAJ,Ben1991,BT76,GGJ2000,OH1966I,OH1966II,OH1966III,Sve1978}. The simple equilibrium points come in three types: sink (attracting), center (rotational), and source (repelling).
There can not be saddle points in the usual sense, since the Cauchy-Riemann equations force the eigenvalues of the Jacobian to be of the form $\lambda=\alpha \pm i \beta$, which can not have real parts with opposite sign.
Near zeros of multiplicity $m>1$, there are $m-1$ attracting and $m-1$ repelling directions and $2(m-1)$ elliptic sectors. See Figure \ref{sependdisk} for an example.
\paragraph{Poles} Another type of singularity for the systems \eqref{objects} is the \emph{pole} at infinity. In more general families, e.g. rational vector fields, there may be more than one pole.  The local dynamics near a pole only depends on the order of the pole. Indeed, the phase portrait of a vector field in the neighborhood of a pole of order $m\in \mathbb{N}$ is shown in Sverdlove \cite{Sve1978} to be topologically equivalent to the phase portrait of the differential equation
\begin{equation}
\label{polenormalform}
\dot{z}=\frac{\lambda}{z^m}, \quad \lambda \in \mathbb{C}^{\ast}, \quad m\in \mathbb{N},
\end{equation}
in a neighborhood of $z=0$, and Garijo, Gasull, and Jarque \cite{GGJ2000} extend the result to conformal (biholomorphic) conjugacy.
Even though holomorphic vector fields can not have saddle points in the usual sense, the local behavior of Equation \eqref{polenormalform} \emph{resembles} that of a saddle point in the following way. It has $m+1$ attracting and $m+1$ repelling straight trajectories meeting at $z=0$, alternating in orientation, and the other trajectories of \eqref{polenormalform} in neighborhood of $z=0$ are hyperbolic in form \cite{Nee1994} (see Figure \ref{sependdisk}).
\paragraph{No limit cycles} A marked difference between holomorphic and general real planar vector fields is that there can not be isolated periodic trajectories, hence no limit cycles, for the former, due to the identity theorem of complex analysis: consider a period $T$ solution $\gamma(t,z_0)$. Since $\gamma(T,z)-z$ is holomorphic in $z$ and $=0$ on an arc of $\gamma$, it is identically $0$ in a neighborhood of $z_0$.\par
\paragraph{Separatrices} The global  structure of the integral curves up to topological equivalence is determined by the \emph{separatrices}, the maximal trajectories that are incoming to  or outgoing from the poles  \cite{DN1975,ALGM1973}. Therefore, unsurprisingly, the separatrices will play a vital role in understanding the bifurcations. The reader may refer to Figure \ref{sependdisk} to guide his or her understanding of the details below which are needed for this paper.
%%%%%
Polynomials of degree $d>2$ have a pole of order $d-2$ at the point at infinity, hence there are $2(d-1)$ separatrices. Separatrices for monic polynomials have asymptotic directions $\ell \pi/(d-1), \ \ell =0,\dots,2d-3$ at infinity. The separatrices come in two types for polynomial vector fields: \emph{landing} separatrices, which have $\alpha$- or $\omega$-limit at an equilibrium point for $t \rightarrow -\infty$ or $t \rightarrow +\infty$ respectively; and \emph{homoclinic} separatrices which join infinity to itself  and are defined for $t$ in a finite interval.
Separatrices $s_{\ell}$ are labelled according to their asymptotic directions. \emph{Outgoing}  separatrices are denoted $s_k$ with odd $k$, \emph{incoming}  separatrices are denoted $s_j$ with even $j$. Homoclinic separatrices are labelled $s_{k,j}$, with odd $k$ and even $j$ corresponding to its outgoing and incoming asymptotic directions at infinity.
Near infinity, the complement of the separatrices has $2(d-1)$ connected components, which define $2(d-1)$ accesses to infinity. The \emph{ends} can informally be thought of as $2(d-1)$ points at infinity, counted as distinct from within each access. Formally, the \emph{ends} $e_{\ell}$, $\ell =0,\dots,2d-3$ are the principal points of the prime ends (in the sense of Carath\'{e}odory) of these regions. The label is chosen such that the access corresponding to $e_{\ell}$ is between the separatrices with labels $\ell -1$ and $\ell$ mod $2(d-1)$ (again, see Figure \ref{sependdisk} for an example).
%%%%%%%%%%%%%%%%
\begin{figure}[htbp]%
%\centering
\large
    %\frame{
    \begin{pspicture}(0,0)(.97\textwidth,5.5)%
    \centering
    %\frame{
    \put(0.5,0){\includegraphics[width=.9\textwidth]{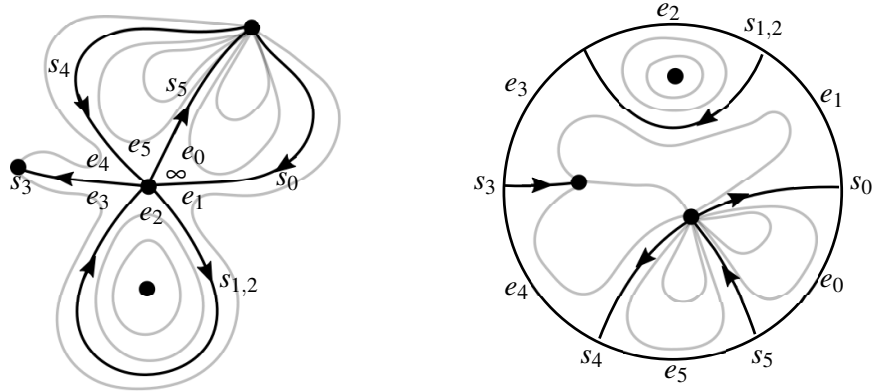}}%
    %}
    \put(2.55,2.8){\color[rgb]{0,0,0}\makebox(0,0)[lb]{\smash{{$\infty$}}}}%
    \put(4,2.75){\color[rgb]{0,0,0}\makebox(0,0)[lb]{\smash{{$s_0$}}}}%
    \put(3.25,1.4){\color[rgb]{0,0,0}\makebox(0,0)[lb]{\smash{{$s_{1,2}$}}}}%
    \put(0.5,2.7){\color[rgb]{0,0,0}\makebox(0,0)[lb]{\smash{{$s_3$}}}}%
    \put(1,4.25){\color[rgb]{0,0,0}\makebox(0,0)[lb]{\smash{{$s_4$}}}}%
    \put(2.55,4){\color[rgb]{0,0,0}\makebox(0,0)[lb]{\smash{{$s_5$}}}}%
    \put(11.55,2.7){\color[rgb]{0,0,0}\makebox(0,0)[lb]{\smash{{$s_0$}}}}%
    \put(10.15,4.8){\color[rgb]{0,0,0}\makebox(0,0)[lb]{\smash{{$s_{1,2}$}}}}%
    \put(6.6,2.7){\color[rgb]{0,0,0}\makebox(0,0)[lb]{\smash{{$s_3$}}}}%
    \put(8,0.4){\color[rgb]{0,0,0}\makebox(0,0)[lb]{\smash{{$s_4$}}}}%
    \put(10.25,0.4){\color[rgb]{0,0,0}\makebox(0,0)[lb]{\smash{{$s_5$}}}}%
    \put(2.75,3.1){\color[rgb]{0,0,0}\makebox(0,0)[lb]{\smash{{$e_0$}}}}%
    \put(2.75,2.5){\color[rgb]{0,0,0}\makebox(0,0)[lb]{\smash{{$e_1$}}}}%
    \put(2.2,2.3){\color[rgb]{0,0,0}\makebox(0,0)[lb]{\smash{{$e_2$}}}}%
    \put(1.5,2.5){\color[rgb]{0,0,0}\makebox(0,0)[lb]{\smash{{$e_3$}}}}%
    \put(1.5,3){\color[rgb]{0,0,0}\makebox(0,0)[lb]{\smash{{$e_4$}}}}%
    \put(2.05,3.2){\color[rgb]{0,0,0}\makebox(0,0)[lb]{\smash{{$e_5$}}}}%
    \put(11.15,1.4){\color[rgb]{0,0,0}\makebox(0,0)[lb]{\smash{{$e_0$}}}}%
    \put(11.15,3.85){\color[rgb]{0,0,0}\makebox(0,0)[lb]{\smash{{$e_1$}}}}%
    \put(9,5){\color[rgb]{0,0,0}\makebox(0,0)[lb]{\smash{{$e_2$}}}}%
    \put(7,4){\color[rgb]{0,0,0}\makebox(0,0)[lb]{\smash{{$e_3$}}}}%
    \put(7,1.3){\color[rgb]{0,0,0}\makebox(0,0)[lb]{\smash{{$e_4$}}}}%
    \put(9.1,0.2){\color[rgb]{0,0,0}\makebox(0,0)[lb]{\smash{{$e_5$}}}}%
\end{pspicture}
%}
  \caption{View of the separatrix graph at $\infty$ (left) and in the disk model (right) for a degree $d=4$ monic and centered complex polynomial vector field. There is one center, one sink, and one double equilibrium point. There is a pole of order $d-2=2$ at $\infty$,  so there are $2(d-1)=6$ separatrix directions, labelled by their asymptotic directions $\ell \pi/3, \ \ell =0,\dots,5$. The separatrices $s_0, \ s_3, \ s_4$, and $s_5$ are \emph{landing} separatrices, and $s_{1,2}$ is a   \emph{homoclinic} separatrix, labelled by its two asymptotic directions at $\infty$. There are 6 \emph{ends} $e_{\ell}$, principal points of the prime ends (in the sense of Carath\'{e}odory) defined by the complement of the separatrices in a neighborhood of infinity.  The ends are labelled such that $e_{\ell}$ is between  $s_{\ell -1}$ and $s_{\ell}$ mod $2(d-1)$.}
   \label{sependdisk}
\end{figure}
%%%%%%%%%%%%%%%%%%%%%%%%%%
%\paragraph{Disk Model}
The separatrix structure can be represented in a
\emph{separatrix disk model}, by blowing up the point at infinity to $\mathbb{S}^1$, labelling the points $\exp \left( \frac{2\pi {\rm i} \ell}{2d-2} \right)$, $\ell=0,\dots,2d-3$ on $\mathbb{S}^1$ by $s_{\ell}$, and embedding the separatrix graph in the disk (see again Figure \ref{sependdisk}).
\subsection{Rectifying Coordinates and Transversals}
The separatrices as explained above provide a decomposition of the dynamical plane into open sets, called \emph{zones}, whose boundaries contain a union of separatrices.   The results in this paper rely on decomposing the dynamical plane into these zones, which are  building blocks that glue together to give a polynomial vector field. We will look at deformations of these glued building blocks to understand the nearby polynomial vector fields, i.e. the bifurcations. The key players here are the \emph{rectifying coordinates}, elaborated on below. Technical details are given in this subsection for self-contained reading, but the main points which are needed for this paper are:
\begin{enumerate}
\item there are three zone types for polynomials: strip, cylinder, and half-plane; and
\item  to understand the statement of Theorem \ref{rank1characterization}, one needs to understand the labelling of separatrices on the boundaries of these strips, cylinders, and half-planes as in Remark \ref{indexrelation}.
\end{enumerate}
\paragraph{Rectifying Coordinates} It is well-known that for polynomials, the connected components of $\mathbb{C}$ minus the separatrix graph (called \emph{zones}) come in three types: center zones, sepal zones, and $\alpha \omega$-zones. In
the last case, there are exactly two singular points on the boundary of the zone, and all trajectories inside the zone have their $\alpha$-limit at one singular point and $\omega$-limit at the other.  These three zone types are isomorphic to half-planes, strips, or cylinders via the \emph{rectifying coordinates} $\Phi(z) = \int_{z_0}^z \frac{dw}{P(w)}$ (called the \emph{distinguished} or \emph{natural parameter} in the literature on quadratic differentials \cite{JAJ}). Under $\Phi$, trajectories are pushed forward to horizontal lines (see Figures \ref{centerrectgen}, \ref{sepalrectgen}, and \ref{striprectgen}). \par

The basin of a center, called a \emph{center zone}, is isomorphic to an upper or lower half-infinite cylinder by the rectifying coordinates $\Phi$ (see Figures \ref{centerrectgen} and \ref{sumdynres}).  Trajectories are pushed forward under $\Phi$ to horizontal circles on the cylinder. The boundary of a center zone consists of a sequence $s_{k_1,j_1}, s_{k_2,j_2}, \dots, s_{k_n,j_n}$ of one or several homoclinic separatrices, together with the ends between them.  A counterclockwise center zone is isomorphic to an upper half-infinite cylinder (as in Figure \ref{centerrectgen}) with the rectified homoclinic separatrices on the lower boundary. The indices of the $s_{k_i,j_i}$ on the boundary satisfy $k_{i+1}=j_{i}+1$, where subindexes $i$ are $i=1,\dots,n$ mod $n$, and indices $k, \ j=0,\dots,2d-3$ are mod $2(d-1)$. A clockwise center zone is isomorphic to a  lower half-infinite cylinder with the sequence of homoclinics on the upper boundary. The indices of the $s_{k_i,j_i}$ on this upper boundary satisfy $k_{i+1}=j_{i}-1$. For a center zone, the first separatrix $s_{k_1,j_1}$ in the sequence is not well-defined, since the sequence of homoclinics closes to a circle. The order is otherwise well-defined by the orientation of the separatrices. \par
%%%%%%%%%%%%%%%%
\begin{figure}[htbp]%
%\centering
\large
    %\frame{
    \begin{pspicture}(0,0)(.97\textwidth,5)%
    \centering
    \put(1.5,0.25){\includegraphics[width=.85\textwidth]{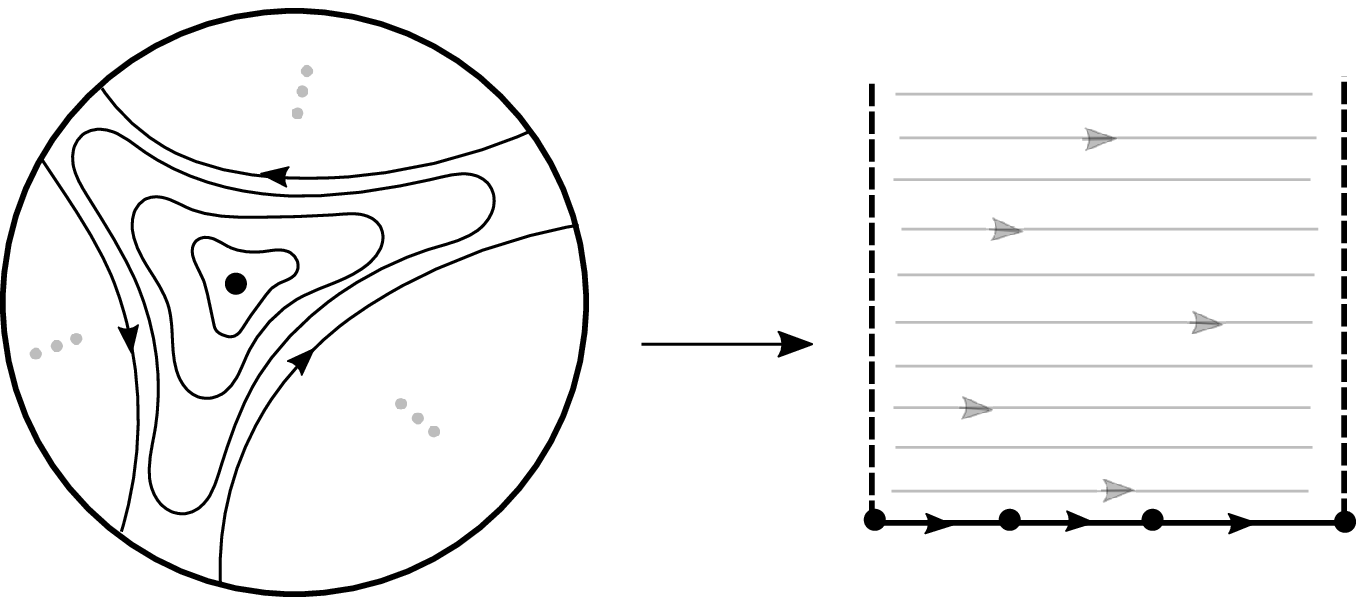}}%
    \put(5.65,3.8){\color[rgb]{0,0,0}\makebox(0,0)[lb]{\smash{{$k_1=j_3+1$}}}}%
    \put(1.9,4.25){\color[rgb]{0,0,0}\makebox(0,0)[lb]{\smash{{$j_1$}}}}%
    \put(0,3.6){\color[rgb]{0,0,0}\makebox(0,0)[lb]{\smash{{$k_2=j_1+1$}}}}%
    \put(2.05,0.3){\color[rgb]{0,0,0}\makebox(0,0)[lb]{\smash{{$j_2$}}}}%
    \put(2.95,0){\color[rgb]{0,0,0}\makebox(0,0)[lb]{\smash{{$k_3=j_2+1$}}}}%
    \put(5.95,3.1){\color[rgb]{0,0,0}\makebox(0,0)[lb]{\smash{{$j_3$}}}}%
    \put(6.9,2.5){\color[rgb]{0,0,0}\makebox(0,0)[lb]{\smash{{$\Phi$}}}}%
    \put(8.4,0.5){\color[rgb]{0,0,0}\makebox(0,0)[lb]{\smash{{$s_{k_1,j_1}$}}}}%
    \put(9.45,0.5){\color[rgb]{0,0,0}\makebox(0,0)[lb]{\smash{{$s_{k_2,j_2}$}}}}%
    \put(10.65,0.5){\color[rgb]{0,0,0}\makebox(0,0)[lb]{\smash{{$s_{k_3,j_3}$}}}}%
\end{pspicture}
%}
  \caption{The basin of a center, called a \emph{center zone} (left), is isomorphic to an upper or lower half-infinite cylinder by the rectifying coordinates $\Phi(z) = \int_{z_0}^z \frac{dw}{P(w)}$ (right). The light gray dots in the disk model on the left represent unspecified dynamics in that region. The vertical dashed lines on the right-hand side are  identified by horizontal translation, giving a half-infinite cylinder. Trajectories are pushed forward under $\Phi$ to horizontal lines (circles, under identification of vertical dashed lines). The boundary of a center zone consists of a sequence $s_{k_1,j_1}, s_{k_2,j_2}, \dots, s_{k_n,j_n}$ of one or several homoclinic separatrices and the \emph{ends} between them.  A counterclockwise center zone gives an upper half-infinite cylinder (as pictured) with the rectified homoclinic separatrices on the lower boundary, whose indices satisfiy  $k_{i+1}=j_{i}+1$, where subindexes are $i=1,\dots,n$ mod $n$, and  $k, \ j=0,\dots,2d-3$ are mod $2(d-1)$. A clockwise center zone gives a lower half-infinite cylinder with the sequence of homoclinics on the upper boundary, whose indices satisfy $k_{i+1}=j_{i}-1$.  The first separatrix $s_{k_1,j_1}$ in the sequence is not well-defined, but the order is otherwise well-defined by the orientation of the separatrices.   }
   \label{centerrectgen}
\end{figure}
A \emph{sepal zone}, i.e. an elliptic sector,  is isomorphic to an upper or lower half-plane by $\Phi$ (see Figure \ref{sepalrectgen}).  Trajectories are pushed forward under $\Phi$ to horizontal lines. The boundary of a sepal zone consists of a sequence of separatrices $s_{j_0}$, $s_{k_1,j_1}, s_{k_2,j_2}, \dots, s_{k_n,j_n}$, $s_{k_0}$  and the ends between them.  The separatrices $s_{j_0}$ and $s_{k_0}$ are landing, and between them are $n\geq 0$ homoclinic separatrices. An upper half-plane  has the rectified separatrices on the lower boundary, whose indices satisfy  $k_{i+1}=j_{i}+1$, $i=0,\dots,n-1$ and $k_0=j_n+1$,  $k_i, \ j_i=0,\dots,2d-3$  mod $2(d-1)$. A lower half-plane (as in Figure \ref{sepalrectgen}) is similar except the index relationship of separatrices on the upper boundary follows the rule $k_{i+1}=j_{i}-1$, $i=0,\dots,n-1$ and $k_0=j_n-1$.  The order of the separatrices on the boundary is well-defined by the orientation of the separatrices in rectifying coordinates, left to right.\par
%%%%%%%%%%%%%%%%%%%%%%%%%%
\begin{figure}[htbp]%
%\centering
\large
    %\frame{
    \begin{pspicture}(0,0)(.98\textwidth,5)%
    \centering
    \put(0.95,0.25){\includegraphics[width=.92\textwidth]{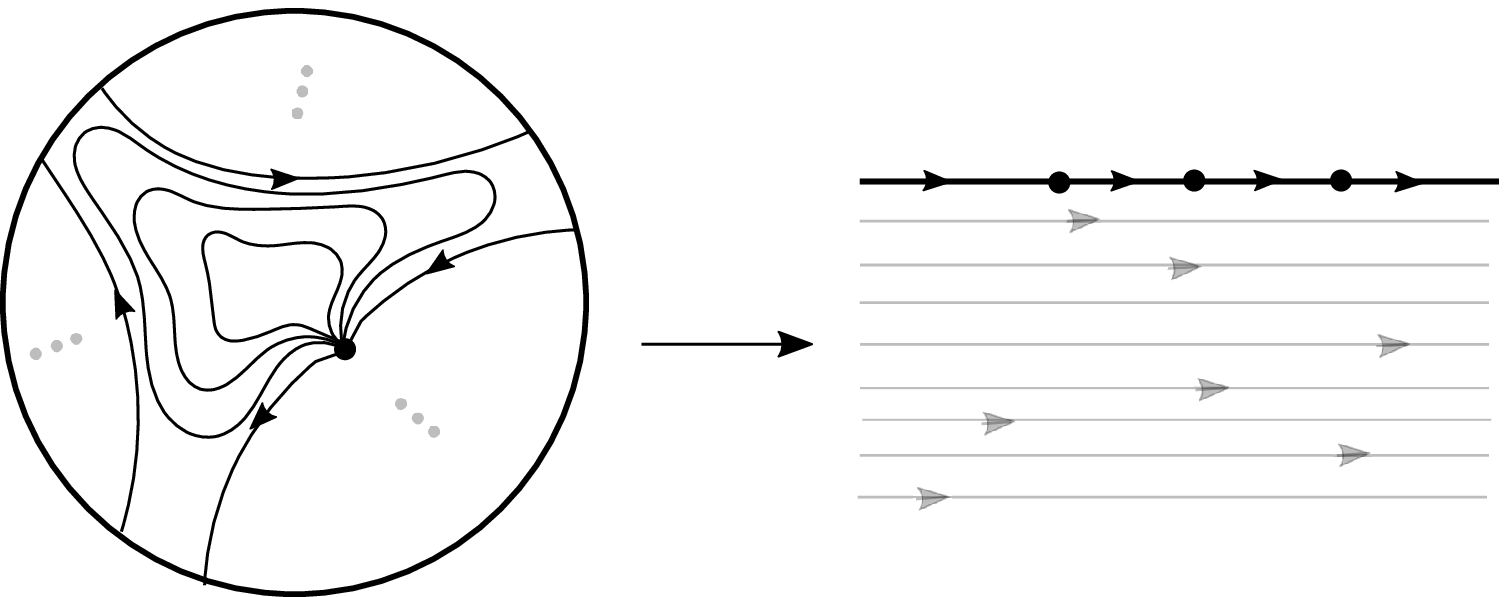}}%
    \put(2.25,0.1){\color[rgb]{0,0,0}\makebox(0,0)[lb]{\smash{{$j_0$}}}}%
    \put(0,0.5){\color[rgb]{0,0,0}\makebox(0,0)[lb]{\smash{{$k_1=j_0-1$}}}}%
    \put(0.8,3.6){\color[rgb]{0,0,0}\makebox(0,0)[lb]{\smash{{$j_1$}}}}%
    \put(0,4.2){\color[rgb]{0,0,0}\makebox(0,0)[lb]{\smash{{$k_2=j_1-1$}}}}%
    \put(5,3.8){\color[rgb]{0,0,0}\makebox(0,0)[lb]{\smash{{$j_2$}}}}%
    \put(5.35,2.95){\color[rgb]{0,0,0}\makebox(0,0)[lb]{\smash{{$k_0$}}}}%
    \put(7.4,3.65){\color[rgb]{0,0,0}\makebox(0,0)[lb]{\smash{{$s_{j_0}$}}}}%
    \put(9,3.65){\color[rgb]{0,0,0}\makebox(0,0)[lb]{\smash{{$s_{k_1,j_1}$}}}}%
    \put(10.1,3.65){\color[rgb]{0,0,0}\makebox(0,0)[lb]{\smash{{$s_{k_2,j_2}$}}}}%
    \put(11.7,3.65){\color[rgb]{0,0,0}\makebox(0,0)[lb]{\smash{{$s_{k_0}$}}}}%
    \put(6.15,2.5){\color[rgb]{0,0,0}\makebox(0,0)[lb]{\smash{{$\Phi$}}}}%
\end{pspicture}
%}
  \caption{A sepal zone, i.e. an elliptic sector,  is isomorphic to an upper or lower half-plane by $\Phi$. The light gray dots in the disk model on the left represent unspecified dynamics in that region. Trajectories are pushed forward under $\Phi$ to horizontal lines. The boundary consists of a sequence of separatrices $s_{j_0}$, $s_{k_1,j_1}, s_{k_2,j_2}, \dots, s_{k_n,j_n}$,  $s_{k_0}$ of two landing separatrices, zero or more homoclinic separatrices, and the \emph{ends} between them.  An upper half-plane  has the rectified separatrices on the lower boundary, whose indices satisfy  $k_{i+1}=j_{i}+1$, $i=0,\dots,n-1$ and $k_0=j_n+1$, with $k_i, \ j_i=0,\dots,2d-3$ mod $2(d-1)$. A lower half-plane (as pictured) has separatrices on the upper boundary whose indices satisfy $k_{i+1}=j_{i}-1$, $i=0,\dots,n-1$ and $k_0=j_n-1$.  The order of the separatrices on the boundary is well-defined by the orientation of the separatrices in rectifying coordinates, left to right.}
   \label{sepalrectgen}
\end{figure}
A region with distinct $\alpha$- and $\omega$-limit points on the boundary is isomorphic to a horizontal strip under $\Phi$ (see Figure \ref{striprectgen}).  Trajectories are pushed forward under $\Phi$ to horizontal lines. The boundary of a strip zone has two components. The upper  component consists of a sequence of separatrices $s^+_{j_0}$, $s^+_{k_1,j_1}, s^+_{k_2,j_2}, \dots, s^+_{k_n,j_n}$,  $s^+_{k_0}$ of two landing separatrices, $n\geq 0$ homoclinic separatrices, and the ends between them.  Similarly, the lower boundary component consists of a sequence of separatrices $s^-_{j_0}$, $s^-_{k_1,j_1}, s^-_{k_2,j_2}, \dots, s^-_{k_m,j_m}$, $s^-_{k_0}$, $m\geq 0$. The relationship between the indices on the upper boundary component is $k^+_{i+1}=j^+_{i}-1$, $i=0,\dots,n-1$ and $k^+_0=j^+_n-1$, and on the lower boundary component is $k^-_{i+1}=j^-_{i}+1$, $i=0,\dots,m-1$ and $k^-_0=j^-_m+1$.  The order of the separatrices on each boundary component is well-defined by the orientation of the separatrices in rectifying coordinates, left to right.\par
%%%%%%%%%%%%%%%%%%%%%%%%%%
\begin{figure}[htbp]%
%\centering
\large
    %\frame{
    \begin{pspicture}(0,0)(.99\textwidth,5.5)%
    \centering
    \put(1.05,0.25){\includegraphics[width=.91\textwidth]{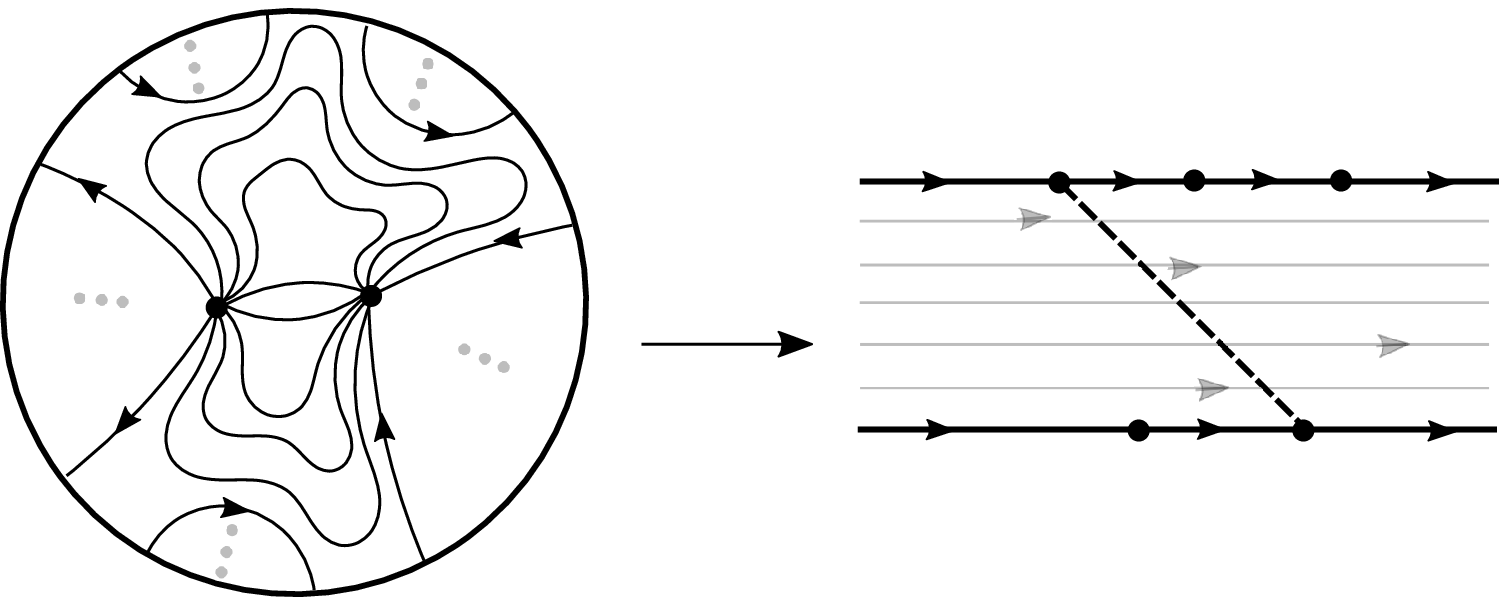}}%
    \put(1,1){\color[rgb]{0,0,0}\makebox(0,0)[lb]{\smash{{$j^-_0$}}}}%
    \put(0.35,0.25){\color[rgb]{0,0,0}\makebox(0,0)[lb]{\smash{{$k^-_1=j^-_0+1$}}}}%
    \put(3.2,0){\color[rgb]{0,0,0}\makebox(0,0)[lb]{\smash{{$j^-_1$}}}}%
    \put(4.15,0.2){\color[rgb]{0,0,0}\makebox(0,0)[lb]{\smash{{$k^-_0=j^-_1+1$}}}}%
    \put(0.9,3.5){\color[rgb]{0,0,0}\makebox(0,0)[lb]{\smash{{$j^+_0$}}}}%
    \put(0,4.2){\color[rgb]{0,0,0}\makebox(0,0)[lb]{\smash{{$k^+_1=j^+_0-1$}}}}%
    \put(2.75,4.75){\color[rgb]{0,0,0}\makebox(0,0)[lb]{\smash{{$j^+_1$}}}}%
    \put(3.75,4.65){\color[rgb]{0,0,0}\makebox(0,0)[lb]{\smash{{$k^+_2=j^+_1-1$}}}}%
    \put(4.85,4){\color[rgb]{0,0,0}\makebox(0,0)[lb]{\smash{{$j^+_2$}}}}%
    \put(5.4,3.1){\color[rgb]{0,0,0}\makebox(0,0)[lb]{\smash{{$k^+_0=j^+_2-1$}}}}%
    \put(7.4,3.6){\color[rgb]{0,0,0}\makebox(0,0)[lb]{\smash{{$s^+_{j_0}$}}}}%
    \put(9,3.6){\color[rgb]{0,0,0}\makebox(0,0)[lb]{\smash{{$s^+_{k_1,j_1}$}}}}%
    \put(10.1,3.6){\color[rgb]{0,0,0}\makebox(0,0)[lb]{\smash{{$s^+_{k_2,j_2}$}}}}%
    \put(11.7,3.6){\color[rgb]{0,0,0}\makebox(0,0)[lb]{\smash{{$s^+_{k_0}$}}}}%
    \put(7.4,1.1){\color[rgb]{0,0,0}\makebox(0,0)[lb]{\smash{{$s^-_{j_0}$}}}}%
    \put(9.75,1.1){\color[rgb]{0,0,0}\makebox(0,0)[lb]{\smash{{$s^-_{k_1,j_1}$}}}}%
    \put(11.7,1.1){\color[rgb]{0,0,0}\makebox(0,0)[lb]{\smash{{$s^-_{k_0}$}}}}%
    \put(6.2,2.4){\color[rgb]{0,0,0}\makebox(0,0)[lb]{\smash{{$\Phi$}}}}%
\end{pspicture}
%}
  \caption{A region with distinct $\alpha$- and $\omega$-limit points on the boundary is isomorphic to a horizontal strip under $\Phi$. The light gray dots in the disk model on the left represent  unspecified dynamics in that region. Trajectories are pushed forward under $\Phi$ to horizontal lines. The boundary of a strip zone has two components. The upper  component consists of a sequence of separatrices $s^+_{j_0}$, $s^+_{k_1,j_1}, s^+_{k_2,j_2}, \dots, s^+_{k_n,j_n}$,  $s^+_{k_0}$ of two landing separatrices, zero or more homoclinic separatrices, and the ends between them.  The lower boundary component consists of a sequence of separatrices $s^-_{j_0}$, $s^-_{k_1,j_1}, s^-_{k_2,j_2}, \dots, s^-_{k_m,j_m}$, $s^-_{k_0}$. The indices on the upper boundary component satisfy $k^+_{i+1}=j^+_{i}-1$, $i=0,\dots,n-1$ and $k^+_0=j^+_n-1$, and the indices on the lower boundary component satisfy  $k^-_{i+1}=j^-_{i}+1$, $i=0,\dots,m-1$ and $k^-_0=j^-_m+1$. The order of the separatrices on each boundary component is well-defined by the orientation of the separatrices in rectifying coordinates, left to right.}
   \label{striprectgen}
\end{figure}
\begin{remark}
\label{indexrelation}
The index relationship of the sequence of separatrices on the boundary of a zone does not depend on the type of zone (center, elliptic, or strip). It only depends on whether the sequence is on an upper boundary or lower boundary.  Omitting details that can be found in the paragraphs above, the relationship can be summarized as: $k_{i+1}=j_i-1$ on an upper boundary, and  $k_{i+1}=j_i+1$ on a lower boundary.
\end{remark}
\paragraph{Transversals}
There are a number of geodesics in $\allowbreak \mathbb{C} \setminus \allowbreak \{ \text{equilibrium points} \}$ that join the point at infinity to itself, with respect to the metric with length element $\frac{|dz|}{|P(z)|}$. This is seen in rectifying coordinates as straight line segments which connect ends to ends. There are $s+h$ among these that are singled out: the $h$ homoclinic separatrices and the  $s$ \emph{distinguished transversals} $T_{k, j}$.
The distinguished transversals are a unique and well-defined choice of one transversal for each of the $s$ strips. This choice is the transversal $T_{k,j}$  which joins $e_k$ to $e_j$, where $e_k$ is the right-most end on the lower boundary of the strip and $e_j$ is the left-most end on the upper boundary of the strip  (see Figures \ref{striprectgen} and \ref{sumdynres}). The reason for this choice is not needed in this paper, but it is explained in detail in \cite{KDcomb}.
%%%%%%%%%%%%%%%%%%%%%%%%%%
\subsection{Classification and Pseudo-Invariants}
%A complete set of realizable invariants for the classification of these vector fields is a metric graph, which is the transversal graph,
\paragraph{Classification}
One way to classify the polynomial vector fields, i.e. define a complete set of realizable invariants, is to express the topological structure via an admissible gluing of strips, cylinders, and half-planes, and the geometry by the \emph{analytic invariants}
\begin{equation}\label{analinvsdef}
  (\underline{\alpha},\underline{\tau})=\left( \int_{T_{k,j}}\frac{dz}{P(z)}, \dots, \int_{s_{k,j}}\frac{dz}{P(z)} \right) \in \HH^s \times \R_+^h,
\end{equation}
where the set of $s_{k,j}$ and $T_{k,j}$ are the homoclinic separatrices and the distinguished transversals for $\xi$. That is, assign $h$ real numbers $\tau=\int_{s_{k,j}}\frac{dz}{P(z)}>0$ to the homoclinics  and $s$ complex numbers $\alpha=\int_{T_{k,j}}\frac{dz}{P(z)}\in \mathbb{H}$ to the distinguished transversals (see \cite{Sent} and \cite{BD09} for more details). The analytic invariants $\tau \in \R_+$ that correspond to homoclinic separatrices are the Euclidean lengths of the rectified homoclinic separatrices on the boundaries of the strips, cylinders, and half-planes; and the analytic invariants $\alpha \in \HH_+$ are complex numbers that record the height and shear of each strip (the complex "length" of $T_{k,j}$). There are other ways to present the classification of polynomial vector fields which will not be explained in this paper. \par
%%%%%%%%%%%%%%%%
\begin{figure}[htbp]%
\large
%\frame{
    \begin{pspicture}(0,0)(12,5)%
    \put(0.3,0){\includegraphics[width=.95\textwidth]{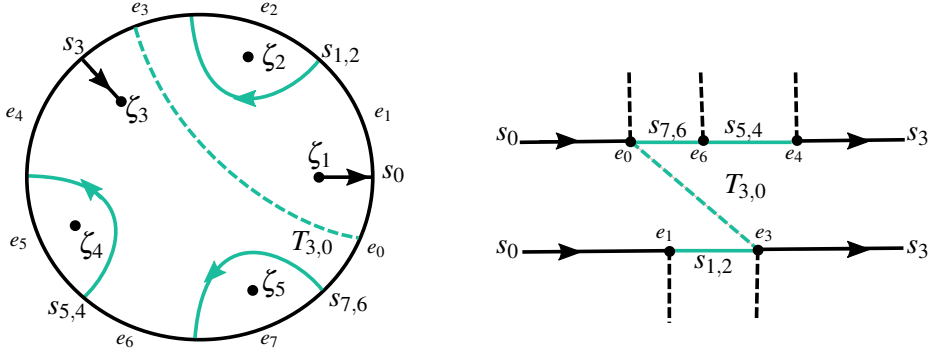}}%
    \put(4,2.4){\color[rgb]{0,0,0}\makebox(0,0)[lb]{\smash{{$\zeta_1$}}}}%
    \put(3.4,3.7){\color[rgb]{0,0,0}\makebox(0,0)[lb]{\smash{{$\zeta_2$}}}}%
    \put(1.6,3){\color[rgb]{0,0,0}\makebox(0,0)[lb]{\smash{{$\zeta_3$}}}}%
    \put(1,1.2){\color[rgb]{0,0,0}\makebox(0,0)[lb]{\smash{{$\zeta_4$}}}}%
    \put(3.4,0.7){\color[rgb]{0,0,0}\makebox(0,0)[lb]{\smash{{$\zeta_5$}}}}%
    \put(5,2.2){\color[rgb]{0,0,0}\makebox(0,0)[lb]{\smash{{$s_0$}}}}%
    \put(4.2,3.8){\color[rgb]{0,0,0}\makebox(0,0)[lb]{\smash{{$s_{1,2}$}}}}%
    \put(0.8,3.9){\color[rgb]{0,0,0}\makebox(0,0)[lb]{\smash{{$s_3$}}}}%
    \put(0.6,0.4){\color[rgb]{0,0,0}\makebox(0,0)[lb]{\smash{{$s_{5,4}$}}}}%
    \put(4.3,0.5){\color[rgb]{0,0,0}\makebox(0,0)[lb]{\smash{{$s_{7,6}$}}}}%
    \put(3.8,1.25){\color[rgb]{0,0,0}\makebox(0,0)[lb]{\smash{{$T_{3,0}$}}}}%
    \put(6.5,1.2){\color[rgb]{0,0,0}\makebox(0,0)[lb]{\smash{{$s_0$}}}}%
    \put(6.5,2.7){\color[rgb]{0,0,0}\makebox(0,0)[lb]{\smash{{$s_0$}}}}%
    \put(11.9,1.2){\color[rgb]{0,0,0}\makebox(0,0)[lb]{\smash{{$s_3$}}}}%
    \put(11.9,2.7){\color[rgb]{0,0,0}\makebox(0,0)[lb]{\smash{{$s_3$}}}}%
    \put(9.1,1){\color[rgb]{0,0,0}\makebox(0,0)[lb]{\smash{{$s_{1,2}$}}}}%
    \put(9.5,2.8){\color[rgb]{0,0,0}\makebox(0,0)[lb]{\smash{{$s_{5,4}$}}}}%
    \put(8.5,2.8){\color[rgb]{0,0,0}\makebox(0,0)[lb]{\smash{{$s_{7,6}$}}}}%
    \put(9.5,2){\color[rgb]{0,0,0}\makebox(0,0)[lb]{\smash{{$T_{3,0}$}}}}%
    %%
    %\normalsize
    \small
    \put(4.8,1.2){\color[rgb]{0,0,0}\makebox(0,0)[lb]{\smash{{$e_0$}}}}%
    \put(4.9,3){\color[rgb]{0,0,0}\makebox(0,0)[lb]{\smash{{$e_1$}}}}%
    \put(3.4,4.4){\color[rgb]{0,0,0}\makebox(0,0)[lb]{\smash{{$e_2$}}}}%
    \put(1.7,4.4){\color[rgb]{0,0,0}\makebox(0,0)[lb]{\smash{{$e_3$}}}}%
    \put(0.05,3){\color[rgb]{0,0,0}\makebox(0,0)[lb]{\smash{{$e_4$}}}}%
    \put(0.1,1.3){\color[rgb]{0,0,0}\makebox(0,0)[lb]{\smash{{$e_5$}}}}%
    \put(1.5,0){\color[rgb]{0,0,0}\makebox(0,0)[lb]{\smash{{$e_6$}}}}%
    \put(3.4,0){\color[rgb]{0,0,0}\makebox(0,0)[lb]{\smash{{$e_7$}}}}%
    \put(8.05,2.45){\color[rgb]{0,0,0}\makebox(0,0)[lb]{\smash{{$e_0$}}}}%
    \put(9.05,2.45){\color[rgb]{0,0,0}\makebox(0,0)[lb]{\smash{{$e_6$}}}}%
    \put(10.3,2.45){\color[rgb]{0,0,0}\makebox(0,0)[lb]{\smash{{$e_4$}}}}%
    \put(8.6,1.4){\color[rgb]{0,0,0}\makebox(0,0)[lb]{\smash{{$e_1$}}}}%
    \put(9.9,1.4){\color[rgb]{0,0,0}\makebox(0,0)[lb]{\smash{{$e_3$}}}}%
\end{pspicture}
%}
  \caption{(Left) The disk model of a degree 5 polynomial vector field, with a sink ($\zeta_3$), a source ($\zeta_1$), and three centers ($\zeta_2, \ \zeta_4, \ \zeta_5$). The two \emph{landing} separatrices ($s_0$ and $s_3$) and three \emph{homoclinic} separatrices ($s_{1,2}$, $s_{5,4}$, $s_{7,6}$), are labelled by their asymptotic directions at infinity. The diagonal dashed line segment $T_{3,0}$ is a \emph{distinguished transversal}, labelled by the \emph{ends} it connects ($e_3$ and $e_0$).
  (Right) The same separatrix configuration in rectifying coordinates. There is one strip, and there are three half-infinite cylinders. The cylinders appear as half-infinite vertical strips (right), where the vertical dashed lines are identified.  The boundary of each zone consists of homoclinic and/or landing separatrices, together with the ends.
  If the separatrices are on the upper (resp. lower) boundary of a strip (as pictured), a half-plane, or a cylinder, then the odd $k$ and even $j$ labels satisfy $k_{i+1}=j_i-1$ (resp. $k_{i+1}=j_i+1$) mod $2(d-1)$, reading left to right. The analytic invariants $\tau \in \R_+$ that correspond to homoclinic separatrices are the Euclidean lengths of the rectified homoclinic separatrices, and the analytic invariant $\alpha \in \HH_+$ corresponds to the complex number that gives the height and shear of the strip.
  }
   \label{sumdynres}
\end{figure}
%%%%%%%%%%%%%%%%%%%%%%%%%%
%Furthermore, the $s$ transversals and $h$ homoclinic separatrices subdivide the dynamical plane into $N$ components, each component containing one equilibrium point, not counting multiplicity (see Figure \ref{sumdynres}).
%%%%%%%%%--------------------
%\begin{remark}
%The elegant classification of complex polynomial vector fields in terms of a metric transversal graph would lead one to expect a description of the bifurcations in terms of this graph, but this is not the choice of description the author made for this paper. Part of the difficulty of this approach stems from the fact that the transversal graph changes discontinuously under perturbation. One may be able to work around this problem by defining an "augmented" transversal graph which does move continuously under perturbation, but (as far as the author is aware) it remains an open problem.
%\end{remark}
%These determine, under rectifying coordinates, a configuration of half-planes, strips, and cylinders, whose boundaries consist of rectified separatrices and where the vector field is $\vf$.\par
The implication of the classification that is required for this paper is that it provides a decomposition into strips, cylinders, and half-planes, identified along their boundaries in the appropriate way, and conversely, any (admissible) gluing of strips, cylinders, and half-planes gives rise to a polynomial vector field \cite{BD09}. This construction of a polynomial vector field from gluings of the building blocks is (briefly) as follows.  The strips, cylinders, and half-planes are glued together to construct a Riemann surface which is conformally equivalent to the Riemann sphere minus punctures, and endowing the strips, cylinders, and half-planes with the vector field $\vf$ leads to a polynomial vector field on Riemann sphere, with equilibrium points at the punctures.
We will understand perturbations (bifurcations) of polynomial vector fields by analyzing perturbations in this construction. While the invariants \eqref{analinvsdef} change discontinuously under perturbation (to be explained), one can show that certain perturbations of the building blocks, preserving the original gluing, give rise to polynomial vector fields that are nearby the initial one. This will be explained in further detail in the paragraph below and in Theorem \ref{analstratastr}.
%%%%%%%%%%%%%%%%%-------------
\paragraph{Pseudo-invariants}
The analysis in this paper relies heavily on deformations of the building blocks (strips, cylinders, and half-planes) in rectifying coordinates. More specifically,
let $\xi_0 \in \Xi_d$  be %a vector field, %in the combinatorial class $\mathcal{C}_0$  which is
the vector field to be perturbed. The classification gives a configuration of half-planes, strips, and cylinders, whose boundaries consist of rectified separatrices and where the vector field is $\vf$. %By definition, the corresponding set of analytic invariants for $\xi_0$ is
%\begin{equation}%\label{}
%  (\underline{\alpha}_0,\underline{\tau}_0)=\left( \int_{T_{k,j}}\frac{dz}{P_0(z)}, \dots, \int_{s_{k,j}}\frac{dz}{P_0(z)} \right) \in \HH^s \times \R_+^h,
%\end{equation}
%where the set of $s_{k,j}$ and $T_{k,j}$ are the homoclinic separatrices and the distinguished transversals for $\xi_0$.
Next, the rectified zones are deformed (as sets) by piecewise linear mappings given by sending horizontal segments corresponding to the rectified $s_{k,j}$ for $\xi_0$ to non-horizontal segments (see Figure \ref{distortedzones}).
%%%%%%%%%%%%%%%%%%%%%%%%%%%%%%%%%%%%%%%%%%%
\begin{figure}[htbp]%
\large
    %\frame{
    \begin{pspicture}(0,0)(12,5.5)%
    \put(1.3,0){\includegraphics[width=.8\textwidth]{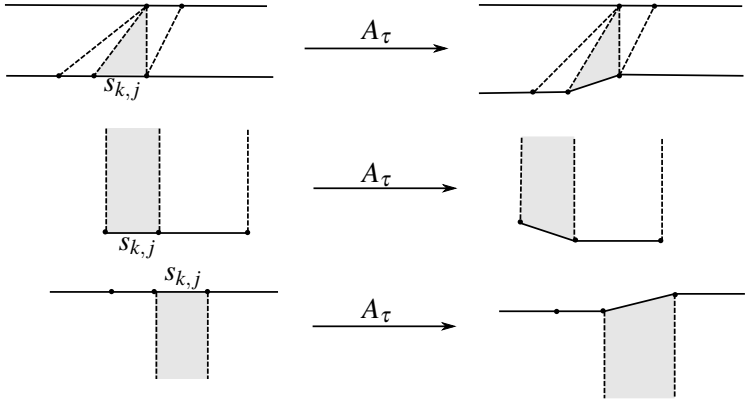}}%
    \put(2.6,4.1){\color[rgb]{0,0,0}\makebox(0,0)[lb]{\smash{$s_{k,j}$}}}%
    \put(2.8,2){\color[rgb]{0,0,0}\makebox(0,0)[lb]{\smash{$s_{k,j}$}}}%
    \put(3.4,1.6){\color[rgb]{0,0,0}\makebox(0,0)[lb]{\smash{$s_{k,j}$}}}%
    \put(6,4.8){\color[rgb]{0,0,0}\makebox(0,0)[lb]{\smash{$A_{\tau}$}}}%
    \put(6,2.9){\color[rgb]{0,0,0}\makebox(0,0)[lb]{\smash{$A_{\tau}$}}}%
    \put(6,1.1){\color[rgb]{0,0,0}\makebox(0,0)[lb]{\smash{$A_{\tau}$}}}%
    \end{pspicture}
    %}
\caption{A vector field $\xi \in \Xi_d$ determines a configuration of half-planes, strips, and cylinders and analytic invariants $(\underline{\alpha}, \underline{\tau})\in \HH^s \times \R_+^h$ which are the lengths on the boundaries. Deforming the rectified zones by piecewise linear mappings corresponds to changes in pseudo-invariants $\left(\underline{ \widetilde{\alpha}} , \underline{\widetilde{\tau}} \right)$ that allow the $\widetilde{\tau}$ to be non-real.}
\label{distortedzones}
\end{figure}
%%%.
As sets, the configuration of the gluing of the rectified zones has not changed, but endowing the distorted zones with $\vf$ will, to be explained by Theorem \ref{analstratastr}, results in another $\xi \in \Xi_d$  close to $\xi_0$, but where $\xi$ has qualitatively different dynamics Even though the configuration of rectified zones has not changed as sets, when the dynamics are considered, the zones have changed type (e.g. Figure \ref{defex0} shows a deformed cylinder that becomes part of a strip zone dynamically).
In terms of $(\underline{\alpha}_0,\underline{\tau}_0)$, this corresponds to allowing one or more of the $\tau_0$ to move off the positive real axis.  However, we can not use $(\underline{\alpha},\underline{\tau})$ to denote the perturbation of $(\underline{\alpha}_0,\underline{\tau}_0)$, since the qualitative change means the distinguished transversals $T_{k,j}$ and homoclinic orbits $s_{k,j}$ change and hence by  definition,  the analytic invariants \eqref{analinvsdef}
%\begin{equation}%\label{}
%  (\underline{\alpha},\underline{\tau})=\left( \int_{T_{k,j}}\frac{dz}{P_{\xi}(z)}, \dots, \int_{s_{k,j}}\frac{dz}{P_{\xi}(z)} \right)
%\end{equation}
will change discontinuously.
This leads us to define the correct quantities to consider to correspond to deformations in rectifying coordinates:
\begin{definition}
The \emph{pseudo-invariants} of $\xi$ with respect to $\xi_0$ are
\begin{equation}\label{pseudoinvsdef}
  \left(\underline{ \widetilde{\alpha}} , \underline{\widetilde{\tau}} \right)=\left( \int_{\gamma_1}\frac{dz}{P_{\xi}(z)}, \dots, \int_{\gamma_{s+h}}\frac{dz}{P_{\xi}(z)} \right)\in \HH^s \times V_{+}^h,
\end{equation}
where $V_{+}$ is an $\epsilon$ neighborhood of $\R_+$, $\gamma_1,\dots,\gamma_{s+h}$ are the curves (as sets) in $\mathbb{C}$ which coincide with the $s$ distinguished transversals and $h$ homoclinic orbits for $\xi_0$ and are fixed under perturbation, and $P_{\xi}$ is a perturbation of $P_{0}$.
\end{definition}
We emphasize that  $\left(\underline{ \widetilde{\alpha}} , \underline{\widetilde{\tau}} \right)$ are in general \emph{not} the analytic invariants for the perturbed vector field (credit is due to the referee  for pointing out this important distinction).
They do correspond to deformations in rectifying coordinates since the initial configuration of half-planes, strips, and cylinders are from the topology of  $\xi_0$, and this configuration as sets has not changed under deformation,  corresponding to fixing the curves $\gamma_1,\dots,\gamma_{s+h}$ under perturbation. Note that the pseudo-invariants are only locally defined.
\section{Parameter Space and Bifurcations}
\label{parmspacesection}
This section reviews  known results on the space $\Xi_d$ of degree $d$ monic, centered complex polynomial vector fields and its decomposition into loci of qualitatively same dynamics. The ideal question to be answered is stated, as well as the restricted question that will be answered in this paper, i.e. describing the  multiplicity-preserving bifurcations. Then we review a result proving the manifold structure of each locus, where the construction of the proof shows that deformations in rectifying coordinates give "nearby" vector fields.
\paragraph{Combinatorial Classes} To understand the possible bifurcations, we examine the space $\Xi_d$ of degree $d$ monic, centered complex polynomial vector fields and partition $\Xi_d$ into \emph{combinatorial classes} $\mathcal{C}$ such that all {$\xi\in\CC$} have the same (labelled) separatrix graph with labelling of the asymptotic directions. Using the (labelled) separatrix graph as the equivalence relation is dynamically meaningful since it implies \emph{topological equivalence} of all $\xi\in\CC$, i.e. there is a homeomorphism which sends trajectories of $\xi\in\CC$ to trajectories of $\xi'\in\CC$ which respects the orientation of but not necessarily the parameterization by time.
Adding the further restriction that the labelling must be the same, there are generally distinct combinatorial classes $\CC_1$ and $\CC_2$ where all $\xi_1\in\CC_1$ are topologically equivalent with all $\xi_2\in\CC_2$.\par
  The space $\Xi_d \cong \C^{d-1}$ since the monic, centered polynomials  $P(z)=z^d+a_{d-2}z^{d-2}+\dots+a_0$ can be parameterized by the coefficients $(a_0,\dots,a_{d-2}) \in \C^{d-1}$.\par
Since the separatrix structure gives topological equivalence of trajectories, a change in separatrix structure under small perturbation therefore gives a bifurcation. Understanding bifurcations is therefore about understanding how the classes $\CC$ fit together in parameter space. This leads us to ask the following question:
\begin{question}
  For every $\xi\in \Xi_d$, what are the  $\CC$ that intersect every neighborhood of $\xi$?
\end{question}
In this paper, the author has chosen to express the answer to this question by showing that each bifurcation is a composition of a finite number of "moves," that is, simpler bifurcations whose types are characterizable.
\paragraph{Stability}
 The \emph{structurally stable} vector fields in $\Xi_d$, i.e. those that do not change qualitatively under  perturbation, are those with neither homoclinic separatrices nor multiple equilibrium points \cite{Sent,TLD09} and are of full dimension in parameter space. The bifurcation locus in $\Xi_d$ consists of all vector fields with at least one homoclinic separatrix or multiple equilibrium point \cite{Sent,TLD09}. One can see this intuitively by imagining a multiple equilibrium point splitting into several equilibrium points or a homoclinic separatrix breaking under perturbation.\par
%\paragraph{Landing Separatrices are Stable}
Bifurcations of \eqref{objects} can only involve breakings of homoclinics and/or splittings of multiple points due to the following result from \cite{TLD09}, which proves that landing separatrices are stable. In other words, an equilibrium point can not "lose" a landing separatrix under small perturbation, as long as its multiplicity is preserved.
\begin{theorem}[Dias, Tan]
\label{landingstablethm}
Let $\zeta^0$ be an equilibrium point for $\xi_0 \in \Xi_d$, and let $\zeta$ be an equilibrium point for $\xi \in \Xi_d$ in a sufficiently small neighborhood of  $\xi_0$ such that $\text{mult}(\zeta^0)=\text{mult}(\zeta)$ and $\lim \limits_{{P} \rightarrow {P_0}}\zeta =\zeta^0$.
If the separatrix $s_{\ell}$ for $\xi_0$ lands at $\zeta^0$, then the separatrix $s_{\ell}$ (same $\ell$) for $\xi$ lands at $\zeta$.
\end{theorem}
\paragraph{Multiplicity-preserving Bifurcations}
  The  bifurcations which allow splitting of multiple equilibrium points are more complicated than the multiplicity-preserving bifurcations, and not only because they involve the variation of more parameters. For splitting bifurcations, the possible changes in \emph{topological} structure may depend on the initial \emph{analytic} data, in addition to the  initial topological data \cite{TLD09}. The bifurcations with this added complexity are to be studied in a future paper; the aim in this paper is to analyze the restricted question:
  \begin{question}
  Given any point $\xi \in \Xi_d$ in the bifurcation locus, what are the possible  bifurcations such that the multiplicities of the equilibrium points are preserved?
  \end{question}
  Such bifurcations are called \emph{multiplicity-preserving bifurcations}. \par
\begin{definition}
Let $\zeta^0_i$, $i=1,\dots, N$ be the equilibrium points for $\xi_0 \in \CC_0$, not counting multiplicity. The \emph{local multiplicity-preserving set of $\xi_0$}, denoted $\mathbf{LMP}(\xi_0)$, is the set of all vector fields $\xi \in \Xi_d \setminus \CC_0$,  in a sufficiently small neighborhood of $\xi_0 \in \Xi_d$ such that $\text{mult} (\zeta_i)=\text{mult} (\zeta^0_i)$, $i=1,\dots,N$.
\end{definition}
\begin{definition}
The \emph{multiplicity-preserving set of $\xi_0$}, denoted $\mathbf{MP}(\xi_0)$, is the set of all combinatorial classes $\CC$ whose intersection with $\mathbf{LMP}(\xi_0)$ is non-empty.
\end{definition}
Note that the notation above differs from that in \cite{TLD09}.
\begin{remark}
If $\xi$ has at least one homoclinic separatrix, then $\mathbf{MP}(\xi)$ is non-empty.
\end{remark}
%%%%----------------
\begin{remark}
Since the multiplicities are preserved, $\mathbf{LMP}(\xi_0)$ can be locally parameterized by the $N$ roots (not counting multiplicity), and in fact can be locally parameterized by $N-1$ roots due to centering of the polynomial.
\end{remark}
\subsection{Deformations Give Bifurcations}
In this subsection, we recall a result from \cite{TLD09}, which proves that each $\mathcal{C}$ is a manifold. While the result as stated is of minor importance for this paper, the construction in the proof implies that piecewise-linear deformations in rectifying coordinates do give "nearby" vector fields in our space $\Xi_d$, and these deformations are (locally) parameterized by the pseudo-invariants \eqref{pseudoinvsdef}.
\paragraph{Parameterizing a Class}
Within a class $\mathcal{C}$, vector fields have the  same zone types (strip, half-plane, cylinder) with the same labelling, but the analytic invariants $\alpha_i\in \mathbb{H}_+$, $i=1,\dots, s$, and $\tau_i\in \mathbb{R}_+$, $i=1,\dots,h$, will be different (see Figure \ref{withinC}).
%%%%%%%%%%%%%%%%%%%%%%%%%%%%%%%%%%%%%%%%%%%
\begin{figure}[htbp]%
\large
    %\frame{
    \begin{pspicture}(0,0)(12,5)%
    \put(1,0){\includegraphics[width=.85\textwidth]{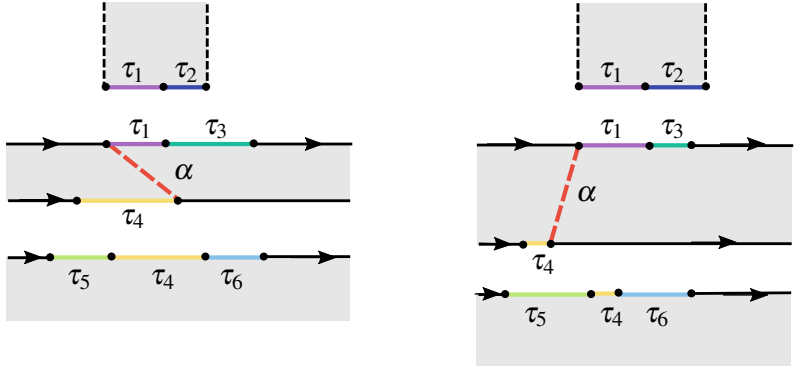}}%
    \put(3.2,2.5){\color[rgb]{0,0,0}\makebox(0,0)[lb]{\smash{{$\alpha$}}}}%
    \put(8.5,2.2){\color[rgb]{0,0,0}\makebox(0,0)[lb]{\smash{{$\alpha$}}}}%
    \put(2.5,3.85){\color[rgb]{0,0,0}\makebox(0,0)[lb]{\smash{{$\tau_1$}}}}%
    \put(8.8,3.85){\color[rgb]{0,0,0}\makebox(0,0)[lb]{\smash{{$\tau_1$}}}}%
    \put(2.6,3.1){\color[rgb]{0,0,0}\makebox(0,0)[lb]{\smash{{$\tau_1$}}}}%
    \put(8.8,3.1){\color[rgb]{0,0,0}\makebox(0,0)[lb]{\smash{{$\tau_1$}}}}%
    \put(3.2,3.85){\color[rgb]{0,0,0}\makebox(0,0)[lb]{\smash{{$\tau_2$}}}}%
    \put(9.6,3.85){\color[rgb]{0,0,0}\makebox(0,0)[lb]{\smash{{$\tau_2$}}}}%
    \put(3.6,3.1){\color[rgb]{0,0,0}\makebox(0,0)[lb]{\smash{{$\tau_3$}}}}%
    \put(9.6,3.1){\color[rgb]{0,0,0}\makebox(0,0)[lb]{\smash{{$\tau_3$}}}}%
    \put(2.5,1.9){\color[rgb]{0,0,0}\makebox(0,0)[lb]{\smash{{$\tau_4$}}}}%
    \put(2.9,1.1){\color[rgb]{0,0,0}\makebox(0,0)[lb]{\smash{{$\tau_4$}}}}%
    \put(7.9,1.3){\color[rgb]{0,0,0}\makebox(0,0)[lb]{\smash{{$\tau_4$}}}}%
    \put(8.8,0.6){\color[rgb]{0,0,0}\makebox(0,0)[lb]{\smash{{$\tau_4$}}}}%
    \put(1.8,1.1){\color[rgb]{0,0,0}\makebox(0,0)[lb]{\smash{{$\tau_5$}}}}%
    \put(7.8,0.6){\color[rgb]{0,0,0}\makebox(0,0)[lb]{\smash{{$\tau_5$}}}}%
    \put(3.8,1.1){\color[rgb]{0,0,0}\makebox(0,0)[lb]{\smash{{$\tau_6$}}}}%
    \put(9.4,0.6){\color[rgb]{0,0,0}\makebox(0,0)[lb]{\smash{{$\tau_6$}}}}%
\end{pspicture}
%}
\caption{Two vector fields $\xi_1 \in \CC$ (left) and $\xi_2 \in \CC$ (right). They have the same zones with the same labelling, but the analytic invariants $\alpha_i\in \mathbb{H}_+$, $i=1,\dots, s$, and $\tau_i\in \mathbb{R}_+$, $i=1,\dots, h$,  are different. }
\label{withinC}
 \end{figure}
%%%%%%%%%%%%%%%%%%%%%%%%%%%%%%%%%%%%%%
This idea is formalized in the following theorem from \cite{TLD09}.
\begin{theorem}[Dias]
\label{analstratastr}
There exists a real analytic isomorphism $F:\HH^s \times \R_+^h\rightarrow \mathcal{C}$, which is $\C$-analytic in the first $s$ coordinates and $\R$-analytic in the last $h$ coordinates. It is the restriction of a holomorphic mapping in $(s+h)$ complex variables: $\widetilde{F} :\HH^s \times V_{+}^h\rightarrow \Xi_d$, where $V_{+}$ is an $\epsilon$ neighborhood of $\R_+$.
\end{theorem}
In particular, $\widetilde{F}:(\underline{\widetilde{\alpha}},\underline{\widetilde{\tau}})\mapsto \xi_{(\underline{\widetilde{\alpha}},\underline{\widetilde{\tau}})}\in \Xi_d$ is holomorphic (in $s+h$ variables) in a neighborhood of $(\underline{\alpha}_0,\underline{\tau}_0)\in \HH^s \times \R_+^h \subset \C^{s+h}$.\par
    The proof of Theorem \ref{analstratastr} is briefly summarized for its use in this manuscript. For details, see \cite{TLD09}.
A vector field $\xi \in \Xi_d$ is  in a combinatorial class $\mathcal{C}$ and has a corresponding set of analytic invariants $(\underline{\alpha}, \underline{\tau})\in \HH^s \times \R_+^h$. This determines a glued configuration of half-planes, strips, and cylinders, in which the vector field is $\vf$.
The rectified zones (as sets) are deformed by piecewise linear mappings given by allowing the $\tau$ to leave the real axis (see Figure \ref{distortedzones}), corresponding to a set of pseudo-invariants $\left(\underline{\widetilde{\alpha}},\underline{\widetilde{\tau}}\right)$.
The gluing gives a Riemann surface conformally isomorphic to the Riemann sphere with punctures. Endowing the deformed zones with $\vf$ leads to a polynomial vector field $\xi_{(\underline{\widetilde{\alpha}},\underline{\widetilde{\tau}})} \in \Xi_d$ on the Riemann sphere with equilibrium points at the punctures. The holomorphic dependence of parameters in the Measurable Riemann Mapping Theorem gives that $\xi_{(\underline{\widetilde{\alpha}},\underline{\widetilde{\tau}})}$ depends holomorphically on the $(\underline{\widetilde{\alpha}},\underline{\widetilde{\tau}})$.
%%%%-------------------------------
\section{Deformations in Rectifying Coordinates}
\label{deformationssection}
The results in the above section show that  deformations in rectifying coordinates correspond to bifurcations in the family $\Xi_d$, in particular, those where homoclinic separatrices break in various ways. We will see in examples below that the resulting change in separatrix structure can be read from the deformed rectifying coordinates. We will prove in this section that these deformations "cover" the local multiplicity-preserving set, meaning that every multiplicity-preserving bifurcation can be seen as a deformation in rectifying coordinates.
\paragraph{Example 1 - The Breaking of one Homoclinic Separatrix}
Consider the example depicted in Figure \ref{defex0}. The top row corresponds to a combinatorial configuration (in the disk and in rectifying coordinates) in $\Xi_3$ which has one sink, one source, and one center.  The separatrices $s_2$ and $s_3$ are \emph{landing}, and there is one homoclinic separatrix $s_{1,0}$. Perturbing the single real analytic invariant $\tau$ to have non-zero imaginary part results in breaking the single homoclinic separatrix.  The center becomes a sink or source, and one can see where the separatrices now land by looking in the rectifying coordinates.\par
%%%%%%%%%%%%%%%%%%%%%%%%%%%%%%%%%%%%%%%%%%%
\begin{figure}[htbp]%
\large
    %\frame{
    \begin{pspicture}(0,0)(12,9.5)%
    \put(1.5,0){\includegraphics[width=.75\textwidth]{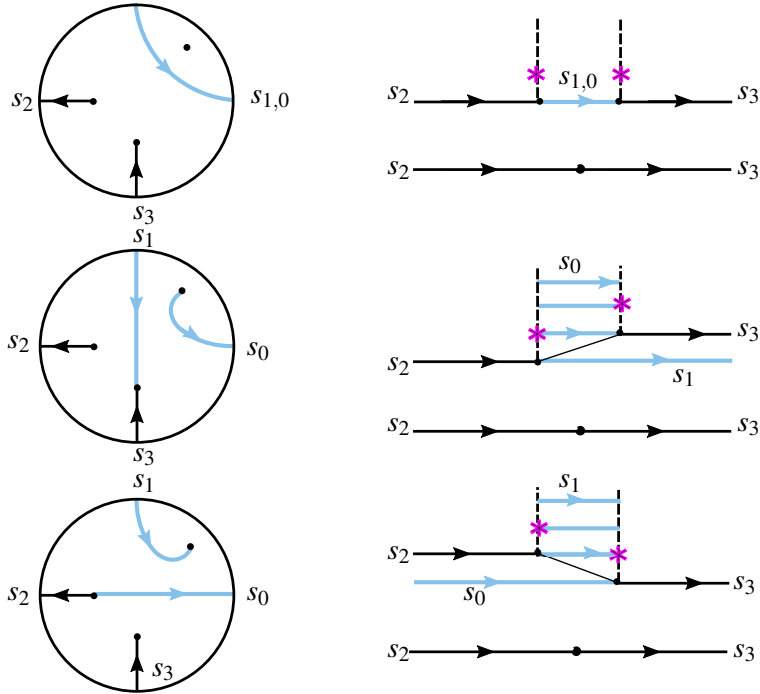}}%
    \put(4.2,7.85){\color[rgb]{0,0,0}\makebox(0,0)[lb]{\smash{ ${s_{1,0}}$}}}%
    \put(1.05,7.8){\color[rgb]{0,0,0}\makebox(0,0)[lb]{\smash{{ ${s_{2}}$}}}}%
    \put(2.65,6.3){\color[rgb]{0,0,0}\makebox(0,0)[lb]{\smash{{ ${s_{3}}$}}}}%
    \put(4.15,4.5){\color[rgb]{0,0,0}\makebox(0,0)[lb]{\smash{{ ${s_{0}}$}}}}%
    \put(2.65,6){\color[rgb]{0,0,0}\makebox(0,0)[lb]{\smash{{ ${s_{1}}$}}}}%
    \put(1,4.6){\color[rgb]{0,0,0}\makebox(0,0)[lb]{\smash{{ ${s_{2}}$}}}}%
    \put(2.65,3.1){\color[rgb]{0,0,0}\makebox(0,0)[lb]{\smash{{ ${s_{3}}$}}}}%
    \put(4.15,1.25){\color[rgb]{0,0,0}\makebox(0,0)[lb]{\smash{{ ${s_{0}}$}}}}%
    \put(2.65,2.75){\color[rgb]{0,0,0}\makebox(0,0)[lb]{\smash{{ ${s_{1}}$}}}}%
    \put(1,1.25){\color[rgb]{0,0,0}\makebox(0,0)[lb]{\smash{{ ${s_{2}}$}}}}%
    \put(2.9,0.25){\color[rgb]{0,0,0}\makebox(0,0)[lb]{\smash{{ ${s_{3}}$}}}}%
    \put(8.25,8.1){\color[rgb]{0,0,0}\makebox(0,0)[lb]{\smash{{ ${s_{1,0}}$}}}}%
    \put(6,6.9){\color[rgb]{0,0,0}\makebox(0,0)[lb]{\smash{{ ${s_{2}}$}}}}%
    \put(6,7.9){\color[rgb]{0,0,0}\makebox(0,0)[lb]{\smash{{ ${s_{2}}$}}}}%
    \put(10.6,6.9){\color[rgb]{0,0,0}\makebox(0,0)[lb]{\smash{{ ${s_{3}}$}}}}%
    \put(10.6,7.9){\color[rgb]{0,0,0}\makebox(0,0)[lb]{\smash{{ ${s_{3}}$}}}}%
    \put(8.25,5.65){\color[rgb]{0,0,0}\makebox(0,0)[lb]{\smash{{ ${s_{0}}$}}}}%
    \put(9.75,4.15){\color[rgb]{0,0,0}\makebox(0,0)[lb]{\smash{{ ${s_{1}}$}}}}%
    \put(6,4.3){\color[rgb]{0,0,0}\makebox(0,0)[lb]{\smash{{ ${s_{2}}$}}}}%
    \put(6,3.5){\color[rgb]{0,0,0}\makebox(0,0)[lb]{\smash{{ ${s_{2}}$}}}}%
    \put(10.6,4.75){\color[rgb]{0,0,0}\makebox(0,0)[lb]{\smash{{ ${s_{3}}$}}}}%
    \put(10.6,3.5){\color[rgb]{0,0,0}\makebox(0,0)[lb]{\smash{{ ${s_{3}}$}}}}%
    \put(7,1.25){\color[rgb]{0,0,0}\makebox(0,0)[lb]{\smash{{ ${s_{0}}$}}}}%
    \put(8.25,2.75){\color[rgb]{0,0,0}\makebox(0,0)[lb]{\smash{{ ${s_{1}}$}}}}%
    \put(6,1.8){\color[rgb]{0,0,0}\makebox(0,0)[lb]{\smash{{ ${s_{2}}$}}}}%
    \put(6,0.5){\color[rgb]{0,0,0}\makebox(0,0)[lb]{\smash{{ ${s_{2}}$}}}}%
    \put(10.55,1.4){\color[rgb]{0,0,0}\makebox(0,0)[lb]{\smash{{ ${s_{3}}$}}}}%
    \put(10.55,0.5){\color[rgb]{0,0,0}\makebox(0,0)[lb]{\smash{{ ${s_{3}}$}}}}%
\end{pspicture}%
%}
\caption{The top row corresponds to a combinatorial configuration (in the disk and in rectifying coordinates) in $\Xi_3$ before perturbation. The equilibrium points are one sink, one source, and one center.  The cylinder corresponding to the basin of the center is depicted in rectifying coordinates as a vertical half-infinite strip, with identification of the dashed lines shown by the asterisks. The separatrices $s_2$ and $s_3$ are \emph{landing}, and there is one homoclinic separatrix $s_{1,0}$. There is one real analytic invariant $\tau$ corresponding to this homoclinic separatrix. The second and third rows of the figure correspond to two possible combinatorial configurations after perturbing the initial. The second (resp. third) row corresponds to allowing $\widetilde{\tau}\in \mathbb{H}_+$ (resp. $\widetilde{\tau}\in \mathbb{H}_-$). In the first case, the perturbed center equilibrium point becomes a source, and in the second case it becomes a sink. Notice that $s_2$ and $s_3$ must continue to land at their respective equilibrium points by Theorem \ref{landingstablethm}. }
\label{defex0}
 \end{figure}
 %%%%%%%%%%%%%%%%%%%%%%%%%%%%%%%%%
We explain the general situation when exactly one homoclinic separatrix $s_{k,j}$ breaks. A homoclinic separatrix $s_{k,j}$ is on the boundary of exactly two zones: one "upper" and one "lower" (in rectifying coordinates).  Allowing the single $\widetilde{\tau}$ corresponding to  $s_{k,j}$ to vary holomorphically causes the separatrices $s_k$ and $s_j$ to land after perturbation.  If $\widetilde{\tau} \in \HH_+$, the separatrix $s_k$  lands at the \eqpt \ on the boundary of the   lower zone, and $s_j$ lands at the \eqpt \ on the boundary of the upper zone.
If $\widetilde{\tau} \in \HH_-$, the separatrix $s_k$  lands at the \eqpt \ on the boundary of the upper zone, and $s_j$ lands at the \eqpt \ on the boundary of the lower zone.
The \eqpt \ at which $s_k$ (resp. $s_j$) lands after perturbation is either a sink (resp. source) or \multeq , depending on whether the  zones having $s_{k,j}$ on the boundary are vertical half-strips or strips in the first case, or in the latter case, a half-plane.
\paragraph{Example 2 - Forming a New Homoclinic Separatrix}
 If  more than one pseudo-invariant $\widetilde{\tau}$ vary at the same time, more complicated things can happen. In particular, new homoclinic separatrices can form. Figure \ref{simplehomjoin} shows an example of a degree $d=4$ combinatorial configuration before perturbation (left) and after perturbation (right), in the disk model (top) and in rectifying coordinates (bottom). Before perturbation, there is one double equilibrium point, receiving two landing separatrices $s_2$ and $s_3$, and two centers, whose boundaries consist of one homoclinic separatrix each: $s_{1,0}$ and $s_{5,4}$ with lengths $\tau_1$ and $\tau_2$ respectively. In rectifying coordinates, this corresponds to two half-planes (only one is shown) and two cylinders (the dashed vertical half lines should be seen as identified under horizontal translation).  Perturbing such that $\Im (\widetilde{\tau_1})+\Im (\widetilde{\tau_2})=0$ (right), $s_{1,0}$ and $s_{5,4}$ both break, and a new homoclinic separatrix $s_{1,4}$ forms. The separatrices $s_0$ and $s_5$ land after perturbation at a source and a sink respectively, which were centers before perturbation.\par
 %%%%%%%%%%%%%%%%%%%%%%%%%%%%%%%%%%%%%%%%%%
\begin{figure}[htbp]%
\large
    %\frame{
    \begin{pspicture}(0,0)(12,7)%
    \put(0.9,0){\includegraphics[width=.85\textwidth]{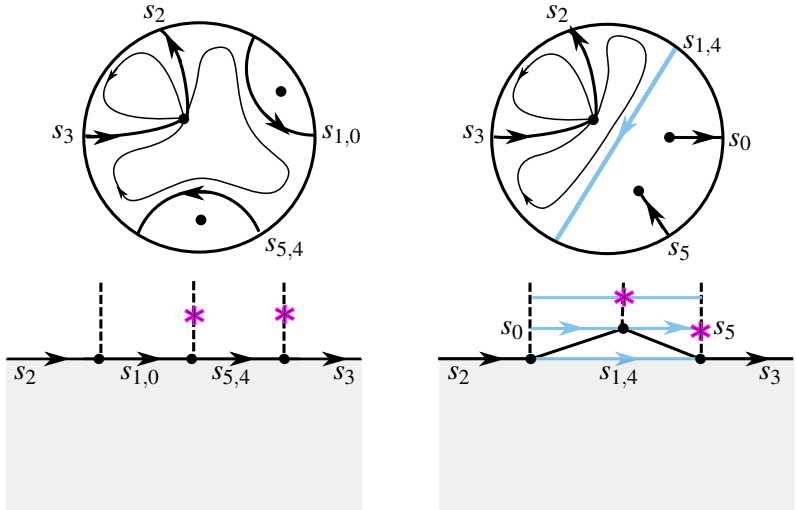}}%
    \put(5.05,5){\color[rgb]{0,0,0}\makebox(0,0)[lb]{\smash{$s_{1,0}$}}}%
    \put(2.7,6.6){\color[rgb]{0,0,0}\makebox(0,0)[lb]{\smash{$s_2$}}}%
    \put(1.5,5){\color[rgb]{0,0,0}\makebox(0,0)[lb]{\smash{$s_3$}}}%
    \put(4.3,3.5){\color[rgb]{0,0,0}\makebox(0,0)[lb]{\smash{$s_{5,4}$}}}%
    \put(2.4,1.8){\color[rgb]{0,0,0}\makebox(0,0)[lb]{\smash{$s_{1,0}$}}}%
    \put(1,1.8){\color[rgb]{0,0,0}\makebox(0,0)[lb]{\smash{$s_2$}}}%
    \put(5.2,1.8){\color[rgb]{0,0,0}\makebox(0,0)[lb]{\smash{$s_3$}}}%
    \put(3.6,1.8){\color[rgb]{0,0,0}\makebox(0,0)[lb]{\smash{$s_{5,4}$}}}%
    \put(10.4,4.9){\color[rgb]{0,0,0}\makebox(0,0)[lb]{\smash{$s_0$}}}%
    \put(9.6,3.45){\color[rgb]{0,0,0}\makebox(0,0)[lb]{\smash{$s_5$}}}%
    \put(9.8,6.2){\color[rgb]{0,0,0}\makebox(0,0)[lb]{\smash{$s_{1,4}$}}}%
    \put(8,6.6){\color[rgb]{0,0,0}\makebox(0,0)[lb]{\smash{$s_2$}}}%
    \put(6.9,5){\color[rgb]{0,0,0}\makebox(0,0)[lb]{\smash{$s_3$}}}%
    \put(7.4,2.4){\color[rgb]{0,0,0}\makebox(0,0)[lb]{\smash{$s_0$}}}%
    \put(10.2,2.4){\color[rgb]{0,0,0}\makebox(0,0)[lb]{\smash{$s_5$}}}%
    \put(8.7,1.8){\color[rgb]{0,0,0}\makebox(0,0)[lb]{\smash{$s_{1,4}$}}}%
    \put(6.7,1.8){\color[rgb]{0,0,0}\makebox(0,0)[lb]{\smash{$s_2$}}}%
    \put(10.8,1.8){\color[rgb]{0,0,0}\makebox(0,0)[lb]{\smash{$s_3$}}}%
  \end{pspicture}%
 % }
\caption{An example of a degree $d=4$ combinatorial configuration before perturbation (left) and after perturbation (right), in the disk model (top) and in rectifying coordinates (bottom). Before perturbation, there is one double equilibrium point, receiving two landing separatrices $s_2$ and $s_3$, and two centers, whose boundaries consist of one homoclinic separatrix each: $s_{1,0}$ and $s_{5,4}$ with lengths $\tau_1$ and $\tau_2$ respectively. In rectifying coordinates, this corresponds to two half-planes (only the lower half-plane is shown) and two cylinders (the dashed vertical half lines should be seen as identified as depicted by the asterisks). Perturbing such that $\Im (\widetilde{\tau}_1)+\Im (\widetilde{\tau}_2)=0$ (right), $s_{1,0}$ and $s_{5,4}$ both break, but a new homoclinic separatrix $s_{1,4}$ forms. The separatrices $s_0$ and $s_5$ land after perturbation at the equilibrium points which used to be centers, and become a source and a sink respectively.}
\label{simplehomjoin}
\end{figure}
%%%%%%%%%%%%%%%%%%%%%%%%%%%%%%%%%%%%%%%%%%%%%
We now describe the general requirement for a homoclinic separatrix to form under  perturbation. In order to understand this situation, we need to define H-chains.
\begin{definition}
\label{HchainDef}
An \emph{H-chain} of length $n$ is a sequence of $n$  homoclinic separatrices $\{ s_{k_i,j_i}\}$, $i=1,\dots,n$,  such that for each $i=1 \leq n-1$, either $k_{i+1}=j_i+1$  or $k_{i+1}=j_i-1$ mod $2(d-1)$ (see Figure \ref{sumdynres}, where $\{ s_{7,6}, s_{5,4}\}$ is an H-chain of length $2$).
\end{definition}
%That is, an H-chain is an oriented product of homoclinic separatrices (see Figure \ref{product}). \par
That is, for each $i=1,\dots,n-1$, the pair $s_{k_i,j_i}$ and $s_{k_{i+1},j_{i+1}}$ are on the boundary of the same zone and are consecutive on the boundary reading left to right in rectifying coordinates (see Figure \ref{product}).
\begin{figure}[htbp]%
\large
    %\frame{
    \begin{pspicture}(0,0)(12,5.5)%
    \put(0.9,0){\includegraphics[width=.85\textwidth]{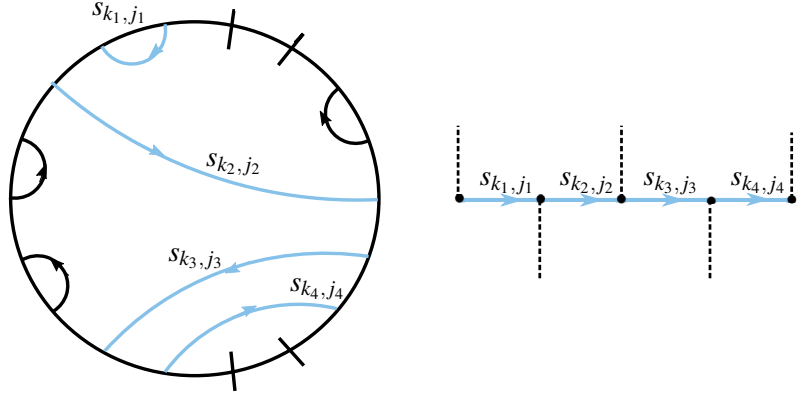}}%
    \put(2,5){\color[rgb]{0,0,0}\makebox(0,0)[lb]{\smash{$s_{k_1,j_1}$}}}%
    \put(3.5,3){\color[rgb]{0,0,0}\makebox(0,0)[lb]{\smash{$s_{k_2,j_2}$}}}%
    \put(3,1.75){\color[rgb]{0,0,0}\makebox(0,0)[lb]{\smash{$s_{k_3,j_3}$}}}%
    \put(4.6,1.35){\color[rgb]{0,0,0}\makebox(0,0)[lb]{\smash{$s_{k_4,j_4}$}}}%
    \put(7.1,2.75){\color[rgb]{0,0,0}\makebox(0,0)[lb]{\smash{$s_{k_1,j_1}$}}}%
    \put(8.15,2.75){\color[rgb]{0,0,0}\makebox(0,0)[lb]{\smash{$s_{k_2,j_2}$}}}%
    \put(9.25,2.75){\color[rgb]{0,0,0}\makebox(0,0)[lb]{\smash{$s_{k_3,j_3}$}}}%
    \put(10.4,2.75){\color[rgb]{0,0,0}\makebox(0,0)[lb]{\smash{$s_{k_4,j_4}$}}}%
  \end{pspicture}%
  %}
\caption{An H-chain is a sequence of homoclinic separatrices $\{ s_{k_i,j_i}\}$, $i=1,\dots,n$,  such that for each $i=1,\dots,n-1$, either $k_{i+1}=j_i+1$  or $k_{i+1}=j_i-1$ mod $2(d-1)$. Shown in the disk model (left) and in rectifying coordinates (right). In the figure, $k_2=j_1+1$, $k_3=j_2-1$, and $k_4=j_3+1$ mod $2(d-1)$. The vertical dashed lines on the right are not (necessarily) to be seen as identified as in Figures \ref{sumdynres}, \ref{defex0}, and \ref{simplehomjoin}. They represent a piece of any of the three types of zones (see Remark \ref{indexrelation}). E.g. $s_{k_1,j_1}$ and $s_{k_2,j_2}$ could be on the lower boundary of a strip, $s_{k_2,j_2}$ and $s_{k_3,j_3}$ could be on the upper boundary of a half-plane, and $s_{k_3,j_3}$ and $s_{k_4,j_4}$ could be on the lower boundary of a cylinder. The same remark holds for the vertical dashed lines in Figures \ref{Hchain1}, \ref{simultaneoushom}, \ref{5to4}, \ref{3to2}, and \ref{Hgraphpath1}.}
\label{product}
\end{figure}
%%%%%%%%%%%%%%%%%%%%%%%%%%%%%% %%%%
The H-chains tell us exactly which homoclinic separatrices can form under  perturbation, as explained in the proposition below.
 \begin{proposition}
 \label{canform}
Let $\xi_0 \in \Xi_d$ have at least two homoclinic separatrices, and let $\xi \in \Xi_d$ be a perturbation of $\xi_0$.  Suppose $s_{k,j_0}$ and $s_{k_0,j}$ are two homoclinic separatrices for $\xi_0$.  The separatrix $s_{k,j}$ for $\xi$, where $k$ and $j$ are the same as in $s_{k,j_0}$ and $s_{k_0,j}$, can form under  perturbation if and only if $s_{k,j_0}$ and $s_{k_0,j}$ belong to a common H-chain where  $s_{k,j_0}$ comes before $s_{k_0,j}$ in the H-chain.
\end{proposition}
For an example, see Figure \ref{simultaneoushom} where $s_{5,4}$ and $s_{13,12}$ have an H-chain in common before perturbation, and $s_{5,12}$  forms after perturbation.
\begin{proof}
Without loss of generality, we may assume the H-chain is of the form
\begin{equation}
s_{k,j_0}=s_{k_1,j_1}, \ s_{k_2,j_2}, \dots , s_{k_n,j_n}=s_{k_0,j}, \quad n\geq 2.
\end{equation}
 For $i=2,\dots,n$, there is a sequence $I_i$ of length $n-1$ with elements in $\{ +, -\}$ corresponding to whether $k_{i+1}=j_i + 1$ or $k_{i+1}=j_i -1$, $i=1,\dots,n-1$ ($I_1$ not defined). If there are $q$ sign changes in this itinerary, then there are $q+1$ zones that $s_k$ needs to pass through to reach $s_j$ (see Figure \ref{Hchain1}).
The separatrix $s_{k,j}$ will form under perturbation if the following conditions on perturbations of the associated $\tau_i$, $i=1,\dots,n$ are satisfied:
For $i=1,\dots,n-1$,
\begin{itemize}
\item if $I_{i+1}=+$, then $\sum \limits_{j=1}^{i} \Im (\widetilde{\tau}_j) <0$;
\item if $I_{i+1}=-$, then $\sum \limits_{j=1}^{i} \Im (\widetilde{\tau}_j) >0$;
\item and $\sum \limits_{j=1}^{n} \Im (\widetilde{\tau}_j) =0$,
\end{itemize} (see again Figure \ref{Hchain1}). These conditions on the partial sums ensure that the new separatrix will not leave a zone before it is desired.\par
If there is no H-chain where $s_{k,j_0}$ comes before $s_{k_0,j}$, then there is no overlapping sequence of zones in rectifying coordinates through which $s_k$ can have access to $s_j$.  The converse then follows from Proposition \ref{covernonsplit} below, which proves that any multiplicity-preserving bifurcation can be realized as a deformation in rectifying coordinates. %\qed
\end{proof}
%%%%%%%%%%%%%%%%%%%%%%%%%%%
\begin{figure}[htbp]%
\large
    %\frame{
    \begin{pspicture}(0,0)(12,2)%
    \put(0.9,0){\includegraphics[width=.85\textwidth]{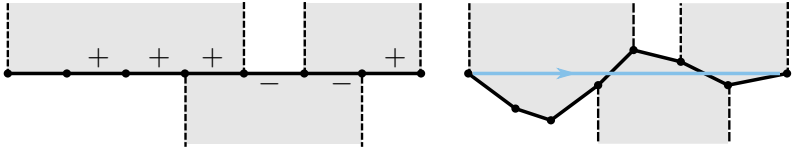}}%
    \put(2,1.1){\color[rgb]{0,0,0}\makebox(0,0)[lb]{\smash{$+$}}}%
    \put(2.8,1.1){\color[rgb]{0,0,0}\makebox(0,0)[lb]{\smash{$+$}}}%
    \put(3.5,1.1){\color[rgb]{0,0,0}\makebox(0,0)[lb]{\smash{$+$}}}%
    \put(4.25,0.75){\color[rgb]{0,0,0}\makebox(0,0)[lb]{\smash{$-$}}}%
    \put(5.2,0.75){\color[rgb]{0,0,0}\makebox(0,0)[lb]{\smash{$-$}}}%
    \put(5.9,1.1){\color[rgb]{0,0,0}\makebox(0,0)[lb]{\smash{$+$}}}%
\end{pspicture}%
%}
\caption{An $H-$chain  $s_{k,j_0}=s_{k_1,j_1}, \ s_{k_2,j_2}, \dots , s_{k_7,j_7}=s_{k_0,j}$. In this example, the sequence $I_i$ for $i=2,\dots,n$ is $I=+,+,+,-,-,+$ ($I_1$ is not defined). A new homoclinic separatrix $s_{k,j}$ can form if the appropriate conditions  are satisfied. In the figure, we require $\widetilde{\tau}_1<0$, $\widetilde{\tau}_1+\widetilde{\tau}_2<0$, $\widetilde{\tau}_1+\widetilde{\tau}_2+\widetilde{\tau}_3<0$, $\widetilde{\tau}_1+\widetilde{\tau}_2+\widetilde{\tau}_3+\widetilde{\tau}_4>0$, $\widetilde{\tau}_1+\widetilde{\tau}_2+\cdots+\widetilde{\tau}_5>0$, $\widetilde{\tau}_1+\widetilde{\tau}_2+\cdots+\widetilde{\tau}_6<0$, and $\widetilde{\tau}_1+\widetilde{\tau}_2+\cdots+\widetilde{\tau}_7=0$. The vertical dashed lines in rectifying coordinates are not (necessarily) to be seen as identified as in Figures \ref{sumdynres}, \ref{defex0}, and \ref{simplehomjoin}. They represent a piece of an unspecified zone that has those homoclinic separatrices on its lower or upper boundary. The same remark holds for the vertical dashed lines in Figures \ref{product},  \ref{simultaneoushom}, \ref{5to4}, \ref{3to2}, and \ref{Hgraphpath1}. }
\label{Hchain1}
\end{figure}
%%%%%%%%%%%%%%%%%%%%
In general, several homoclinic separatrices can form simultaneously under  perturbation, if the conditions on the partial sums as in Proposition \ref{canform} are all satisfied (see Figure \ref{simultaneoushom}).
%%%%%%%%%%%%%%%%%%%%%%%%%%%%%%%%%%
\begin{figure}[htbp]%
\large
    %\frame{
    \begin{pspicture}(0,0)(12,8.5)%
    \put(0.7,0){\includegraphics[width=.9\textwidth]{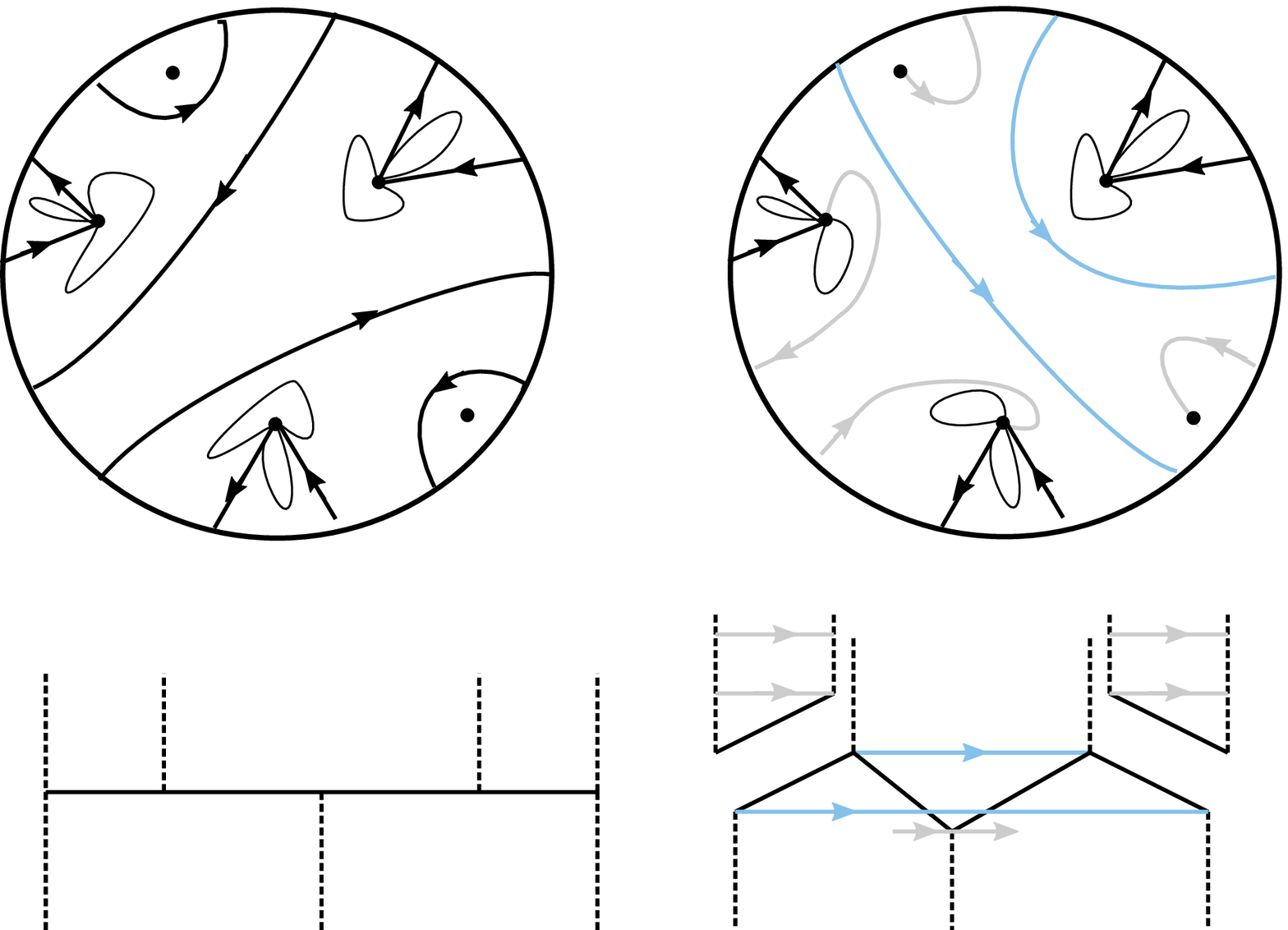}}%
    \put(5.3,6.6){\color[rgb]{0,0,0}\makebox(0,0)[lb]{\smash{$s_1$}}}%
    \put(11.55,6.55){\color[rgb]{0,0,0}\makebox(0,0)[lb]{\smash{$s_1$}}}%
    \put(4.5,7.5){\color[rgb]{0,0,0}\makebox(0,0)[lb]{\smash{$s_2$}}}%
    \put(10.6,7.6){\color[rgb]{0,0,0}\makebox(0,0)[lb]{\smash{$s_2$}}}%
    \put(0.6,6.65){\color[rgb]{0,0,0}\makebox(0,0)[lb]{\smash{$s_6$}}}%
    \put(6.85,6.7){\color[rgb]{0,0,0}\makebox(0,0)[lb]{\smash{$s_6$}}}%
    \put(0.3,5.7){\color[rgb]{0,0,0}\makebox(0,0)[lb]{\smash{$s_7$}}}%
    \put(6.5,5.75){\color[rgb]{0,0,0}\makebox(0,0)[lb]{\smash{$s_7$}}}%
    \put(2.4,3.1){\color[rgb]{0,0,0}\makebox(0,0)[lb]{\smash{$s_{10}$}}}%
    \put(8.5,3.1){\color[rgb]{0,0,0}\makebox(0,0)[lb]{\smash{$s_{10}$}}}%
    \put(3.5,3.15){\color[rgb]{0,0,0}\makebox(0,0)[lb]{\smash{$s_{11}$}}}%
    \put(9.75,3.2){\color[rgb]{0,0,0}\makebox(0,0)[lb]{\smash{$s_{11}$}}}%
    \put(3.5,8){\color[rgb]{0,0,0}\makebox(0,0)[lb]{\smash{$s_{3,8}$}}}%
    \put(2.5,0.9){\color[rgb]{0,0,0}\makebox(0,0)[lb]{\smash{$s_{3,8}$}}}%
    \put(1.1,7.5){\color[rgb]{0,0,0}\makebox(0,0)[lb]{\smash{$s_{5,4}$}}}%
    \put(1.3,0.9){\color[rgb]{0,0,0}\makebox(0,0)[lb]{\smash{$s_{5,4}$}}}%
    \put(1.1,3.6){\color[rgb]{0,0,0}\makebox(0,0)[lb]{\smash{$s_{9,0}$}}}%
    \put(4,0.9){\color[rgb]{0,0,0}\makebox(0,0)[lb]{\smash{$s_{9,0}$}}}%
    \put(4.5,3.5){\color[rgb]{0,0,0}\makebox(0,0)[lb]{\smash{$s_{13,12}$}}}%
    \put(4.95,0.9){\color[rgb]{0,0,0}\makebox(0,0)[lb]{\smash{$s_{13,12}$}}}%
    \put(9.75,8){\color[rgb]{0,0,0}\makebox(0,0)[lb]{\smash{$s_{3,0}$}}}%
    \put(8.75,1.75){\color[rgb]{0,0,0}\makebox(0,0)[lb]{\smash{$s_{3,0}$}}}%
    \put(7.35,7.6){\color[rgb]{0,0,0}\makebox(0,0)[lb]{\smash{$s_{5,12}$}}}%
    \put(7.2,0.7){\color[rgb]{0,0,0}\makebox(0,0)[lb]{\smash{$s_{5,12}$}}}%
    \put(8.75,8){\color[rgb]{0,0,0}\makebox(0,0)[lb]{\smash{$s_4$}}}%
    \put(7,2.2){\color[rgb]{0,0,0}\makebox(0,0)[lb]{\smash{$s_4$}}}%
    \put(6.7,4.6){\color[rgb]{0,0,0}\makebox(0,0)[lb]{\smash{$s_8$}}}%
    \put(8.25,0.6){\color[rgb]{0,0,0}\makebox(0,0)[lb]{\smash{$s_8$}}}%
    \put(7.4,3.7){\color[rgb]{0,0,0}\makebox(0,0)[lb]{\smash{$s_9$}}}%
    \put(9.1,0.6){\color[rgb]{0,0,0}\makebox(0,0)[lb]{\smash{$s_9$}}}%
    \put(11.5,4.5){\color[rgb]{0,0,0}\makebox(0,0)[lb]{\smash{$s_{13}$}}}%
    \put(10.4,2.2){\color[rgb]{0,0,0}\makebox(0,0)[lb]{\smash{$s_{13}$}}}%
    \end{pspicture}%
    %}
\caption{(Left) An initial separatrix configuration in the disk model and in rectifying coordinates. (Right) One possible bifurcation where two new homoclinic separatrices form simultaneously. The vertical dashed lines in rectifying coordinates are not (necessarily) to be seen as identified as in Figures \ref{sumdynres}, \ref{defex0}, and \ref{simplehomjoin}. They represent a piece of an unspecified zone that has those homoclinic separatrices on its lower or upper boundary. The same remark holds for the vertical dashed lines in Figures \ref{product},  \ref{Hchain1}, \ref{5to4}, \ref{3to2}, and \ref{Hgraphpath1}.}
\label{simultaneoushom}
\end{figure}
%%%%%%%%%%%%%%%%%%%%%%%%%%%%
\subsection{Deformations in Rectifying Coordinates Cover the Local Multiplicity-preserving Set}
We prove that all multiplicity-preserving bifurcations can be seen by deformations in the rectifying coordinates. In other words, we want to prove that varying the $\widetilde{\tau}$ covers all multiplicity-preserving bifurcations.
 \begin{proposition}\label{covernonsplit}
  Let $N$ be the number of equilibrium points of $\xi_0$, not counting multiplicity. The map %\newline
  $(\widetilde{\alpha}_1,\dots,\widetilde{\alpha}_s,\widetilde{\tau}_1,\dots,\widetilde{\tau}_h)\mapsto (\zeta_1,\dots, \zeta_{N-1})\in \mathbf{LMP}(\xi_0)$, the local multiplicity-preserving set at $\xi_0$, is locally surjective.
\end{proposition}
\begin{proof}
Due to centering, only $N-1$ roots are independent, hence one can locally parameterize the local multiplicity-preserving set by $N-1=s+h$ of the roots.
 It is enough to show that the map $(\zeta_1,\dots, \zeta_{N-1})\mapsto (\widetilde{\alpha}_1,\dots,\widetilde{\alpha}_s,\widetilde{\tau}_1,\dots,\widetilde{\tau}_h)$ is (locally) well-defined and continuous for all $(\zeta_1,\dots, \zeta_{N-1})\in \mathbf{LMP}(\xi_0)$, the local multiplicity-preserving set of $\xi_0$. This is trivially true by noting that by Definition \ref{pseudoinvsdef},
\begin{equation}%\label{}
  (\widetilde{\alpha}_1,\dots,\widetilde{\alpha}_s,\widetilde{\tau}_1,\dots,\widetilde{\tau}_h)=\left( \int_{\gamma_1}\frac{dz}{P_{\xi}(z)}, \dots, \int_{\gamma_{s+h}}\frac{dz}{P_{\xi}(z)} \right).
\end{equation}
 These integrals are locally well-defined and depend continuously on the $(\zeta_1,\dots, \zeta_{N-1})$.\par
%%----------------------------
%\qed
\end{proof}
Therefore, to understand the multiplicity-preserving bifurcations, it is enough to analyze deformations in rectifying coordinates. %While the proof of Proposition \ref{covernonsplit} is trivial, the reader needs to understand that this detail has been accounted for.
\section{Multiplicity-preserving Bifurcations Can Be Decomposed into Compositions of Rank 1 Bifurcations}
\label{rank1compsection}
%----------------------------
%---------------------------
 In this section, it is  proved that every multiplicity-preserving bifurcation  can be realized as a composition of simpler bifurcations, the \emph{rank 1} bifurcations (to be defined).  The rank 1 bifurcations will be characterized in Section \ref{characterizerank1}, showing they can not be arbitrarily complicated. \par
\subsection{Dimension and Codimension of a Class}
Note that Theorem \ref{analstratastr} gives the (real) dimension of a combinatorial class.
\begin{proposition}\label{dimcodim}
For a combinatorial class $\mathcal{C}\in \Xi_d$, with $s$ strips, $h$ homoclinic separatrices, and $m^{\ast}=\sum_i (\text{mult}(\zeta_i)-1)$,
\begin{equation}
\dim_{\mathbb{R}}(\mathcal{C})=2s+h, \quad \text{and} \quad \codim_{\mathbb{R}}(\mathcal{C})=2m^{\ast}+h.
\end{equation}
\end{proposition}
%%---------------
\begin{proof}
Since $\Xi_d \simeq \mathbb{C}^{d-1}$, it follows $\dim_{\R}(\Xi_d)=2d-2$.  A combinatorial class $\mathcal{C}$ is analytically isomorphic to  $\HH^s \times \R_+^h$, giving $\dim_{\R}(\mathcal{C})=2s+h$.  Let $N$ be the number of equilibrium points not counting multiplicity.   Then $m^{\ast}=d-N$. Putting this together with $s+h=N-1$ gives $s=d-m^{\ast}-1-h$. Therefore, the (real) codimension of each class as a subset of parameter space is
\begin{align}
\codim_{\R}(\mathcal{C})&=2d-2-(2s+h) \nonumber \\
&=2d-2-2(d-m^{\ast}-1-h)-h \nonumber \\
&=2m^{\ast}+h.
\end{align}
\end{proof}
%%%------------------------
\begin{definition}
The \emph{boundary $\partial \CC$ of a combinatorial class $\CC$} is defined in the usual sense: the closure $\overline{\CC}$ in $\Xi_d$, minus $\CC$.
\end{definition}
%%---------------------------
The following lemma proves a statement crucial to the proof of Theorem \ref{rank1comp}.  Namely, for multiplicity-preserving sets, if one class intersects the boundary of another, then it must be entirely contained in the boundary.  This does not generally hold for bifurcations which allow splitting of multiple points and would not be directly applicable to that case.
%%-------------------------
\begin{lemma}
\label{bdycontain}
Let $\xi_0 \in \mathcal{C}_0 \subset \Xi_d$ have at least one homoclinic separatrix so that $\mathbf{MP}(\xi_0)$ is non-empty.
For every $\mathcal{C}\subset \mathbf{MP}(\xi_0)$, it holds that $\mathcal{C}_0 \subset \partial \mathcal{C}$.
\end{lemma}
\begin{proof}
Necessarily $\mathcal{C}_0 \cap \partial \mathcal{C}\neq \emptyset$.  Then  $\mathcal{C}_0 \subset \partial\mathcal{C}$, since by Theorem \ref{analstratastr} the combinatorics of the multiplicity-preserving bifurcations do not depend on initial homoclinic length, but only the relative imaginary parts of the perturbed $(\widetilde{\alpha}_1,\dots,\widetilde{\alpha}_s,\widetilde{\tau}_1,\dots,\widetilde{\tau}_h)$. %\qed
\end{proof}
\begin{corollary}
\label{lessdim}
It follows  that $\CC_0$ has strictly greater codimension than all $\CC \subset \mathbf{MP}(\xi_0)$.
\end{corollary}
%-----------------------------
\subsection{Rank $k$ Bifurcations}
The following notation will be used: a subscript of $0$ denotes before perturbation, and a subscript of $1$ denotes after perturbation.
\begin{definition}\label{rankkbif}
A \emph{rank $k$ bifurcation} is a bifurcation from $\xi_0\in\mathcal{C}_0$ to $\xi_1 \in \mathcal{C}_1$ such that $\dim_{\R}(\mathcal{C}_1)-\dim_{\R}(\mathcal{C}_0)=k$.
\end{definition}
\begin{remark}
The author prefers the terminology "rank"  over "codimension" of a bifurcation, since bifurcations of complex polynomial vector fields that allow splitting of multiple points (to be analyzed in a future paper) can have "rank 0" bifurcations, i.e. two loci with the same dimension can be adjacent, and using the terminology "codimension 0" would suggest that no bifurcation occurs in this case.
\end{remark}
The rank $k$ bifurcations can be loosely characterized as follows by using the codimension. Recall that for all vector fields in a combinatorial class $\CC$, $s$ is the number of strips, $h$ is the  number of homoclinic separatrices, and $m^{\ast}=\sum_i(\text{mult}(\zeta_i)-1)$, where $s, \ h, \ m^{\ast} \in \{ 0, \dots, d-1\}$. Proposition \ref{dimcodim} states that $\dim_{\R} (\CC)=2s+h$ and $\codim_{\R}(\CC)=2m^{\ast}+h$.
%-----------------------------
For analyzing multiplicity-preserving bifurcations, the codimension proves more useful. Using Proposition \ref{dimcodim}, a rank $k$ bifurcation is when $2(m^{\ast}_0-m^{\ast}_1)+(h_0-h_1)=k$. Since  multiplicities are constant, i.e. $m^{\ast}_0=m^{\ast}_1$, this simplifies to $h_0-h_1=k$. In words, the number of homoclinic separatrices before bifurcation is $k$ greater than the number of homoclinic separatrices after bifurcation. Therefore, a rank $k$ multiplicity-preserving bifurcation is such that $h \leq h_0$ homoclinic separatrices break and $h-k$ new homoclinic separatrices form.\par
%---------------------------
The behavior can be further narrowed down by the following proposition.
\begin{proposition}
There are no rank $k \leq 0$ multiplicity-preserving bifurcations.
\end{proposition}
\begin{proof}
 There can not be more homoclinic separatrices after perturbation than existed initially, i.e. $h_1 \leq h_0$. This is because there is a fixed number $2(d-1)$ of asymptotic directions and landing separatrices are stable when multiplicities are fixed by Theorem \ref{landingstablethm}.  Therefore,  there can be no rank $k < 0$ multiplicity-preserving bifurcations. There can be no rank $k=0$ multiplicity-preserving bifurcations, since Corollary \ref{lessdim} implies that the dimension of $\mathcal{C}_0$ is strictly less than the dimension of $\mathcal{C}_1$. %\qed
\end{proof}
%%------------------------------------
\subsection{Compositions of Rank 1 Bifurcations}
By Definition \ref{rankkbif}, a rank 1 bifurcation is a bifurcation from $\xi_0\in\mathcal{C}_0$ to $\xi_1 \in \mathcal{C}_1$ such that $\dim_{\R}(\mathcal{C}_1)-\dim_{\R}(\mathcal{C}_0)=1$. The multiplicity-preserving, rank 1 bifurcations are the bifurcations such that $h_0-h_1=1$, meaning that $h\leq h_0$ homoclinic separatrices break and $h-1$ new ones form. See Figures \ref{defex0}, \ref{simplehomjoin}, \ref{5to4}, and \ref{3to2} for examples.
%%%%%%%%%%%%%%%%%%%%%%%%%%%%%%%
\begin{figure}[htbp]%
\large
    %\frame{
    \begin{pspicture}(0,0)(12,7)%
        \put(0.5,0.25){\includegraphics[width=.95\textwidth]{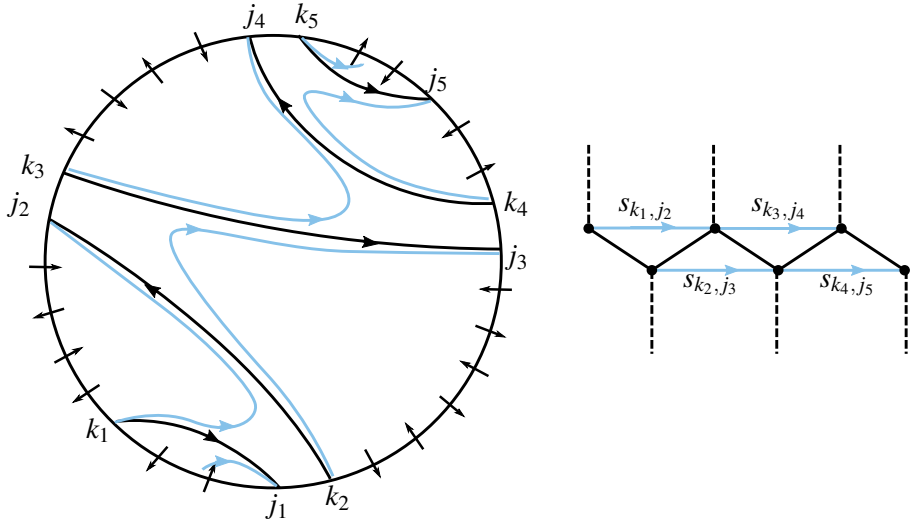}}%
        \put(1.25,1){\color[rgb]{0,0,0}\makebox(0,0)[lb]{\smash{$k_1$}}}%
        \put(3.6,0){\color[rgb]{0,0,0}\makebox(0,0)[lb]{\smash{$j_1$}}}%
        \put(4.4,0.1){\color[rgb]{0,0,0}\makebox(0,0)[lb]{\smash{$k_2$}}}%
        \put(0.2,4){\color[rgb]{0,0,0}\makebox(0,0)[lb]{\smash{$j_2$}}}%
        \put(0.4,4.5){\color[rgb]{0,0,0}\makebox(0,0)[lb]{\smash{$k_3$}}}%
        \put(6.75,3.3){\color[rgb]{0,0,0}\makebox(0,0)[lb]{\smash{$j_3$}}}%
        \put(6.75,4){\color[rgb]{0,0,0}\makebox(0,0)[lb]{\smash{$k_4$}}}%
        \put(3.3,6.5){\color[rgb]{0,0,0}\makebox(0,0)[lb]{\smash{$j_4$}}}%
        \put(4,6.5){\color[rgb]{0,0,0}\makebox(0,0)[lb]{\smash{$k_5$}}}%
        \put(5.7,5.65){\color[rgb]{0,0,0}\makebox(0,0)[lb]{\smash{$j_5$}}}%
        \put(8.25,4){\color[rgb]{0,0,0}\makebox(0,0)[lb]{\smash{$s_{k_1,j_2}$}}}%
        \put(9.1,3){\color[rgb]{0,0,0}\makebox(0,0)[lb]{\smash{$s_{k_2,j_3}$}}}%
        \put(10,4){\color[rgb]{0,0,0}\makebox(0,0)[lb]{\smash{$s_{k_3,j_4}$}}}%
        \put(10.9,3){\color[rgb]{0,0,0}\makebox(0,0)[lb]{\smash{$s_{k_4,j_5}$}}}%
    \end{pspicture}%
    %}
\caption{An example of a rank 1 bifurcation in the disk model (left) and in rectifying coordinates (right). The black arcs represent an H-chain $s_{k_1,j_1}$, $s_{k_2,j_2}$, $s_{k_3,j_3}$, $s_{k_4,j_4}$, and $s_{k_5,j_5}$ that a vector field has before perturbation.  The small black arrows on the rest of the circle indicate that there are other separatrices between, with unspecified behavior. The blue curves represent a change in the separatrix configuration after perturbation. In this example, the five homoclinic separatrices $s_{k_1,j_1}$, $s_{k_2,j_2}$, $s_{k_3,j_3}$, $s_{k_4,j_4}$, and $s_{k_5,j_5}$  break, and four new homoclinics form: $s_{k_1,j_2}$, $s_{k_2,j_3}$, $s_{k_3,j_4}$, and $s_{k_4,j_5}$. The vertical dashed lines in rectifying coordinates are not (necessarily) to be seen as identified as in Figures \ref{sumdynres}, \ref{defex0}, and \ref{simplehomjoin}. They represent  a piece of an unspecified zone that has those homoclinic separatrices on its lower or upper boundary. The same remark holds for the vertical dashed lines in Figures \ref{product},  \ref{Hchain1}, \ref{simultaneoushom}, \ref{3to2}, and \ref{Hgraphpath1}.}
\label{5to4}
\end{figure}
%%%%%%%%%%%%%%%%%%%%%%%%%%%%%%%%%%%%%%
%%%%%%%%%%%%%%%%%%%%%%%%%%%%%%
\begin{figure}[htbp]%
\large
    %\frame{
    \begin{pspicture}(0,0)(12,7)%
        \put(0.5,0.25){\includegraphics[width=.9\textwidth]{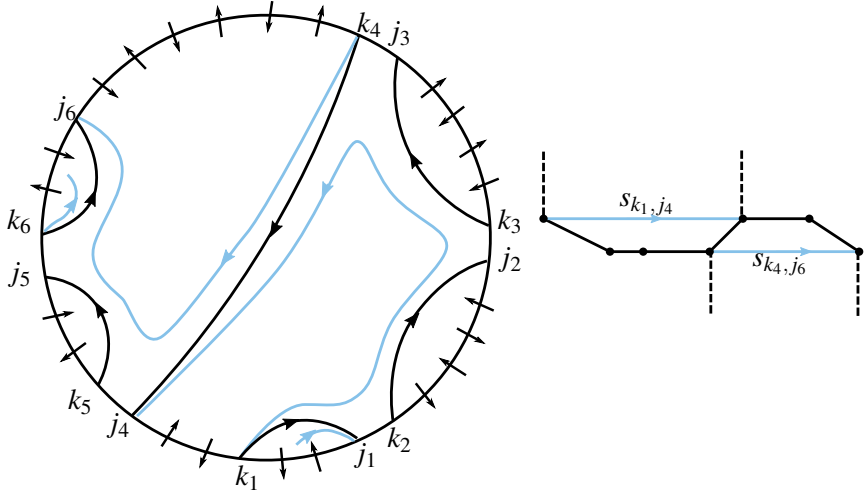}}%
        \put(3.2,0.2){\color[rgb]{0,0,0}\makebox(0,0)[lb]{\smash{$k_1$}}}%
        \put(4.75,0.5){\color[rgb]{0,0,0}\makebox(0,0)[lb]{\smash{$j_1$}}}%
        \put(5.2,0.7){\color[rgb]{0,0,0}\makebox(0,0)[lb]{\smash{$k_2$}}}%
        \put(6.65,3.1){\color[rgb]{0,0,0}\makebox(0,0)[lb]{\smash{$j_2$}}}%
        \put(6.65,3.6){\color[rgb]{0,0,0}\makebox(0,0)[lb]{\smash{$k_3$}}}%
        \put(5.2,6.1){\color[rgb]{0,0,0}\makebox(0,0)[lb]{\smash{$j_3$}}}%
        \put(4.8,6.2){\color[rgb]{0,0,0}\makebox(0,0)[lb]{\smash{$k_4$}}}%
        \put(1.5,0.9){\color[rgb]{0,0,0}\makebox(0,0)[lb]{\smash{$j_4$}}}%
        \put(1,1.2){\color[rgb]{0,0,0}\makebox(0,0)[lb]{\smash{$k_5$}}}%
        \put(0.2,2.9){\color[rgb]{0,0,0}\makebox(0,0)[lb]{\smash{$j_5$}}}%
        \put(0.2,3.6){\color[rgb]{0,0,0}\makebox(0,0)[lb]{\smash{$k_6$}}}%
        \put(0.8,5.1){\color[rgb]{0,0,0}\makebox(0,0)[lb]{\smash{$j_6$}}}%
        \put(8.25,3.9){\color[rgb]{0,0,0}\makebox(0,0)[lb]{\smash{$s_{k_1,j_4}$}}}%
        \put(10,3.1){\color[rgb]{0,0,0}\makebox(0,0)[lb]{\smash{$s_{k_4,j_6}$}}}%
    \end{pspicture}%
    %}
\caption{An example of a rank 1 bifurcation in the disk model (left) and in rectifying coordinates (right). The black arcs represent an H-chain $s_{k_1,j_1}$, $s_{k_2,j_2}$, $s_{k_3,j_3}$, $s_{k_4,j_4}$, $s_{k_5,j_5}$, and $s_{k_6,j_6}$ that a vector field has before perturbation.  The small black arrows on the rest of the circle indicate that there are other separatrices between, with unspecified behavior. The blue curves represent a change in the separatrix configuration after perturbation. In this example, the three homoclinic separatrices $s_{k_1,j_1}$, $s_{k_4,j_4}$, and $s_{k_6,j_6}$  break, and two new homoclinics form: $s_{k_1,j_4}$ and  $s_{k_4,j_6}$. There are homoclinic separatrices in the initial H-chain that did not break under perturbation: $s_{k_2,j_2}$, $s_{k_3,j_3}$, and $s_{k_5,j_5}$. The vertical dashed lines in rectifying coordinates are not (necessarily) to be seen as identified as in Figures \ref{sumdynres}, \ref{defex0}, and \ref{simplehomjoin}. They represent a piece of an unspecified zone that has those homoclinic separatrices on its lower or upper boundary. The same remark holds for the vertical dashed lines in Figures \ref{product},  \ref{Hchain1}, \ref{simultaneoushom}, \ref{5to4}, and \ref{Hgraphpath1}.}
\label{3to2}
\end{figure}
%%%%%%%%%%%%%%%%%%%%%%%%%%%%%%%%%%%%%%
\begin{theorem}
\label{rank1comp}
Every multiplicity-preserving bifurcation can be realized as a composition of rank 1 bifurcations.
\end{theorem}
%%---------------------------
The idea of the proof is to show that every multiplicity-preserving bifurcation is accessible through a sequence of rank 1, multiplicity-preserving bifurcations. To do this, we will show the existence of a sequence of classes such that
\begin{enumerate}
\item all are contained in the multiplicity-preserving set,
\item each class is contained in the boundary of the next, and
\item each differs in dimension by 1 (see Figure \ref{cube}).
\end{enumerate}
%The idea of the proof is to show that for every $\mathcal{C}\subset \mathbf{MP}(\xi_0)$, there exists a sequence of classes $\mathcal{C}_i$,  $i=1,\dots,k-1$ such that $\mathcal{C}_i \subset \mathbf{MP}(\xi_0)$ and $\partial \mathcal{C}\supset \mathcal{C}_{k-1}, \dots, \partial \mathcal{C}_2\supset \mathcal{C}_1, \partial \mathcal{C}_1 \supset \mathcal{C}_0$ and $\mathcal{C}_{i-1}$ to $\mathcal{C}_{i}$  and $\mathcal{C}_{k-1}$ to $\mathcal{C}$ are rank 1 bifurcations (see Figure \ref{cube}).
%%%---------------------------
\begin{figure}[htbp]%
\large
    %\frame{
    \begin{pspicture}(0,0)(12,3)%
        \put(4,0){\includegraphics[width=4cm]{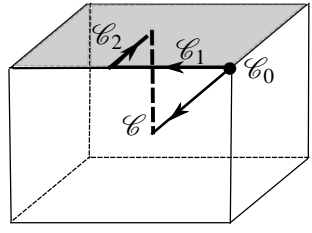}}%
        \put(5.5,1.2){\color[rgb]{0,0,0}\makebox(0,0)[lb]{\smash{$\mathcal{C}$}}}%
        \put(7.1,1.95){\color[rgb]{0,0,0}\makebox(0,0)[lb]{\smash{$\mathcal{C}_0$}}}%
        \put(6.2,2.2){\color[rgb]{0,0,0}\makebox(0,0)[lb]{\smash{$\mathcal{C}_1$}}}%
        \put(5.1,2.4){\color[rgb]{0,0,0}\makebox(0,0)[lb]{\smash{$\mathcal{C}_2$}}}%
    \end{pspicture}
   % }
   \caption{The idea of the proof of Theorem \ref{rank1comp} is to show that for every $\mathcal{C}\subset \mathbf{MP}(\xi_0)$, there exists a sequence of classes $\mathcal{C}_i, \quad i=1,\dots,k-1$ such that $\mathcal{C}_i \subset \mathbf{MP}(\xi_0)$ and $\partial \mathcal{C}\supset \mathcal{C}_{k-1}, \dots, \partial \mathcal{C}_2\supset \mathcal{C}_1, \partial \mathcal{C}_1 \supset \mathcal{C}_0$ and $\mathcal{C}_{i-1}$ to $\mathcal{C}_{i}$ and $\mathcal{C}_{k-1}$ to $\mathcal{C}$ are rank 1 bifurcations. The figure is schematic only and not representative of the actual geometry of the combinatorial classes. }
 \label{cube}
\end{figure}
%%%%%%%%%%%%%%%%%%%%%%%%%
We first prove that if $\mathcal{C}\subset \mathbf{MP}(\xi_0)$, then all of $\mathcal{C}$'s relevant boundary components are also in the multiplicity-preserving set $\mathbf{MP}(\xi_0)$.
\begin{lemma}
\label{nonsplitclosed}
For all $\mathcal{C}\subset \mathbf{MP}(\xi_0)$, it holds that every component $\mathcal{C}_1$ of $\partial \mathcal{C}$ with $\xi_0\in \partial \mathcal{C}_1$, $\mathcal{C}_1\subset \mathbf{MP}(\xi_0)$.
\end{lemma}
\begin{proof}
It will prove to be almost trivially true by the observation that equilibrium points can not coalesce under small perturbation.
The contrapositive is proved: for every $\mathcal{C} \subset \Xi_d$ with $\xi_0 \in \partial \mathcal{C}$, and for every $\mathcal{C}_1$ such that $\mathcal{C}_1 \cap \partial \mathcal{C}\neq \emptyset$ and $\xi_0\in \partial \mathcal{C}_1$, if $\mathcal{C}_1 \not\subset \mathbf{MP}(\xi_0)$, then $\mathcal{C} \not\subset \mathbf{MP}(\xi_0)$. Assume $\mathcal{C}_1$ intersects the boundary of $\mathcal{C}$, and they both have $\xi_0$ in the boundary.  If $\mathcal{C}_1$ is in the splitting set of $\xi_0$, then it means that any vector field in $\mathcal{C}_1$ has more zeros than $\xi_0$. Since $\mathcal{C}_1$ intersects (or is contained in) the boundary of $\mathcal{C}$, then there are points of $\mathcal{C}$ arbitrarily close to  points in $\mathcal{C}_1$. Hence the number of zeros for a vector field in $\mathcal{C}$ must be greater than or equal to the number of zeros of a vector field in $\mathcal{C}_1$ since equilibrium points can not merge under perturbation. Therefore, the number of zeros in $\mathcal{C}$ are greater than the number in $\xi_0$, proving that $\mathcal{C}$ belongs to the splitting set of $\xi_0$. %\qed
\end{proof}
Now we can finish the main proof.
\begin{proof}[Proof of Theorem \ref{rank1comp}]
We will show that for every $\mathcal{C}\subset \mathbf{MP}(\xi_0)$, there exists a sequence of classes $\mathcal{C}_i$,  $i=1,\dots,k-1$ such that $\mathcal{C}_i \subset \mathbf{MP}(\xi_0)$ and $\partial \mathcal{C}\supset \mathcal{C}_{k-1}, \dots, \partial \mathcal{C}_2\supset \mathcal{C}_1, \partial \mathcal{C}_1 \supset \mathcal{C}_0$ and $\mathcal{C}_{i-1}$ to $\mathcal{C}_{i}$  and $\mathcal{C}_{k-1}$ to $\mathcal{C}$ are rank 1 bifurcations. Specifically, it will be proved that
for every $\mathcal{C}$ with $\xi_0 \in \partial \mathcal{C}$ and $\dim_{\R}(\mathcal{C})-\dim_{\R}(\mathcal{C}_0)=k$, $k>1$, there exists $\mathcal{C}_1\subset \mathbf{MP}(\xi_0)$ such that:
\begin{enumerate}
\item[1.] $\dim_{\R}(\mathcal{C})>\dim_{\R}(\mathcal{C}_1)>\dim_{\R}(\mathcal{C}_0)$, and
\item[2.] $\mathcal{C}_1 \subset \partial \mathcal{C}$.
\end{enumerate}
Let $V_{0}$ be  a neighborhood of $\xi_0$ in $\Xi_d$. The set $\mathcal{C}_0$ can not be a dense subset of $\partial \mathcal{C} \cap V_{0}$ by the assumption that $\dim_{\R}(\mathcal{C})-\dim_{\R}(\mathcal{C}_0)>1$.
There must therefore be some other $\mathcal{C}_1$ such that $\partial \mathcal{C} \cap \mathcal{C}_1\neq \emptyset$ and $\xi_0 \in \partial \mathcal{C}_1$.
Next, $\mathcal{C}_1\subset \mathbf{MP}(\xi_0)$ by Lemma \ref{nonsplitclosed},  and item 2. follows from Lemma \ref{bdycontain}. Finally, $\mathcal{C}_1\subset \partial \mathcal{C}$ implies $\dim_{\R}(\mathcal{C}_1)<\dim_{\R}(\mathcal{C})$ by Corollary \ref{lessdim}. %\qed
\end{proof}
%%----------------------------
\section{Characterization of the Rank 1 Multiplicity-preserving Bifurcations}
\label{characterizerank1}
%%%%----------------------------
The result of Theorem \ref{rank1comp} is not helpful if the rank 1 bifurcations can be arbitrarily complicated. We characterize the rank 1 bifurcations in this section to show that this is not the case. It will be proved that all rank 1 bifurcations are of essentially the same type as the examples in Figures \ref{5to4} and \ref{3to2}. We first characterize that type.
\begin{definition}
A \emph{chained homoclinic breaking} is a bifurcation where a sequence $s_{k_1,j_1},...,s_{k_n,j_n}$ of $n\geq 1$ homoclinic separatrices  break such that
\begin{enumerate}
  \item for each $i=1,\dots,n-1$, $s_{k_i,j_i}$ and $s_{k_{i+1},j_{i+1}}$ are on the boundary of the same zone before perturbation,
  \item $n-1$ homoclinic separatrices form: $s_{k_1,j_2},s_{k_2,j_3}, ... , s_{k_{n-1},j_n}$,
  \item $s_{j_1}$ lands at the $\alpha$-limit (equilibrium) point on the boundary of the zone to the right of $s_{k_1,j_1}$, and $s_{k_n}$ lands at the $\omega$-limit (equilibrium) point on the boundary of the zone to the left of $s_{k_n,j_n}$ if $\widetilde{\tau}_{k_1,j_1}\in \mathbb{H}_-$ (similar for $\widetilde{\tau}_{k_1,j_1}\in \mathbb{H}_+$), and
  \item the separatrix graph is otherwise unchanged.
\end{enumerate}
\end{definition}
Interestingly, one can also see that this type of bifurcation inserts a strip zone with the new homoclinics on its boundary (see Figures \ref{5to4} and \ref{3to2}). \par
We can now present the second main theorem of this paper.
\begin{theorem}
\label{rank1characterization}
Every rank 1, multiplicity-preserving  bifurcation is %of the type where a sequence $s_{k_1,j_1},...,s_{k_n,j_n}$ of $n\geq 1$ homoclinic separatrices  break such that for each $i=1,\dots,n-1$, $s_{k_i,j_i}$ and $s_{k_{i+1},j_{i+1}}$ are on the boundary of the same zone, and $n-1$ homoclinic separatrices form: $s_{k_1,j_2},s_{k_2,j_3}, ... , s_{k_{n-1},j_n}$ (see e.g. Figures \ref{5to4} and \ref{3to2}).
a chained homoclinic breaking.
\end{theorem}
In order to prove this theorem, we will build new objects: homoclinic graphs (\emph{H-graphs}, for short). These graphs are to be defined such that an admissible path in the H-graph corresponds to a homoclinic separatrix that can appear under  perturbation, and the vertices of this path correspond to the homoclinic separatrices that \emph{must} break to form that new homoclinic.
\subsection{H-graphs}
%%%----------------------------
H-graphs are directed graphs, based on the initial homoclinic separatrix configuration, to help organize the possible multiplicity-preserving bifurcations.  They are to be defined such that:
\begin{itemize}
  \item an \emph{admissible path} (to be defined) corresponds to a homoclinic separatrix that can appear under  perturbation, and
  \item the vertices of this path correspond to the homoclinic separatrices that \emph{must} break in order for the new homoclinic to form.
\end{itemize}
E.g.  a single admissible  path using $k+1$ vertices corresponds to a rank $k$ bifurcation, since $k+1$ homoclinics break and only one new one forms.
%----------------------------
%%%%%%%%%%%%%%%%%%%%%%%%%%%%%%%%
\begin{figure}[htbp]%
\large
    %\frame{
    \begin{pspicture}(0,0)(12,6)%
        \put(0.3,0){\includegraphics[width=.95\textwidth]{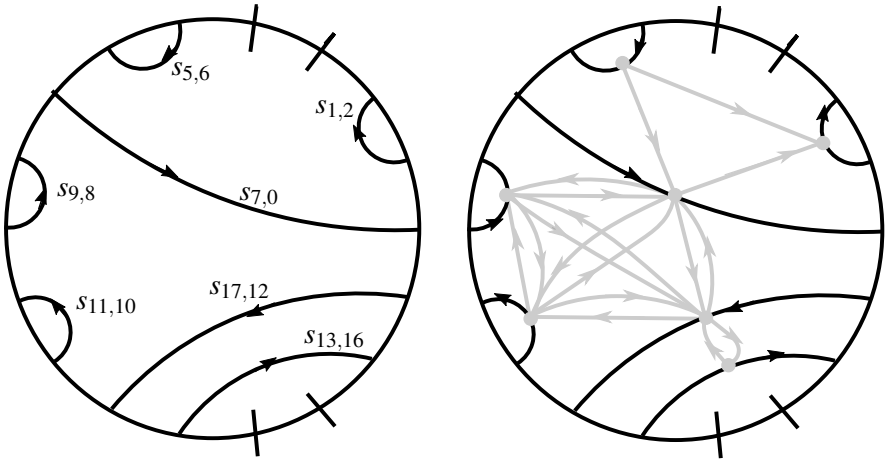}}%
        \put(3.4,3.45){\color[rgb]{0,0,0}\makebox(0,0)[lb]{\smash{$s_{7,0}$}}}%
        \put(4.4,4.6){\color[rgb]{0,0,0}\makebox(0,0)[lb]{\smash{$s_{1,2}$}}}%
        \put(2.5,5.1){\color[rgb]{0,0,0}\makebox(0,0)[lb]{\smash{$s_{5,6}$}}}%
        \put(1,3.5){\color[rgb]{0,0,0}\makebox(0,0)[lb]{\smash{$s_{9,8}$}}}%
        \put(1.25,2){\color[rgb]{0,0,0}\makebox(0,0)[lb]{\smash{$s_{11,10}$}}}%
        \put(3,2.2){\color[rgb]{0,0,0}\makebox(0,0)[lb]{\smash{$s_{17,12}$}}}%
        \put(4.25,1.55){\color[rgb]{0,0,0}\makebox(0,0)[lb]{\smash{$s_{13,16}$}}}%
    \end{pspicture}
    %}
   \caption{(Left) The initial separatrix configuration in the disk model; only the homoclinic separatrices are labelled. (Right) The H-graph embedded in the disk model. Looking at the component with $s_{5,6}$, $s_{7,0}$, and $s_{1,2}$ on the boundary, there is an edge directed from $s_{5,6}$ to $s_{1,2}$ since there is an H-chain $s_{5,6}$, $s_{7,0}$, $s_{1,2}$ which contains both and where $s_{5,6}$ comes before $s_{1,2}$. For this same component, there is not an edge directed from $s_{1,2}$ to $s_{5,6}$ since there is no H-chain   which contains both where $s_{1,2}$ comes before $s_{5,6}$. Looking at the component in the middle,  there are two directed edges between $s_{7,0}$ and $s_{17,12}$ since there is an H-chain $s_{7,0}$, $s_{17,12}$ which gives the edge $s_{7,0}\rightarrow s_{17,12}$, and there is an H-chain $s_{17,12}$, $s_{11,10}$, $s_{9,8}$, $s_{7,0}$ which gives the edge $s_{17,12}\rightarrow s_{7,0}$.  }
 \label{Hgraph}
\end{figure}
%%---------------------------
\begin{definition}
An \emph{H-graph} (see Figure \ref{Hgraph}) corresponding to a combinatorial class $\mathcal{C}$ is a directed graph which is embedded in the combinatorial disk model such that:
\begin{enumerate}
  \item The vertices of the H-graph correspond to the homoclinic separatrices for $\mathcal{C}$.
  \item Each edge is contained in a connected component of the disk minus the homoclinic separatrices
  \item There is an edge directed from $v_1$ to $v_2$ if there is an H-chain containing both corresponding homoclinic separatrices, where the homoclinic for $v_1$ comes before the homoclinic for $v_2$.
\end{enumerate}
\end{definition}
%\begin{remark}
%In the disk model, each connected component of the disk minus the homoclinic separatrices corresponds to a strip, half-plane, or cylinder zone, so leaving a connected component corresponds to leaving a zone.
%\end{remark}
\begin{remark}\label{directionofedges}
Within a connected component that does not correspond to a cylinder, all of the directed edges of the H-graph in that component must run in the same direction as the H-chain on the boundary (see the component in Figure \ref{Hgraph}  with $s_{5,6}$, $s_{7,0}$, and $s_{1,2}$ on the boundary).  In a component that does correspond to a cylinder, some edges in the H-graph may run contrary to the orientation of the H-chain, since one can 'wrap around' the cylinder (see Figure \ref{Hgraph} where the directed edge $s_{7,0}\rightarrow s_{17,12}$ runs in the same direction as the H-chain along the boundary, and the directed edge $s_{17,12}\rightarrow s_{7,0}$ runs contrary to the orientation of the H-chain on the boundary).
\end{remark}
\paragraph{Admissible Paths} Consider the case where a homoclinic separatrix can form from homoclinics on the boundary of a single zone.  Though more homoclinics \emph{may} break under perturbation, only two are \emph{required} to break for this new homoclinic to form (see Figure \ref{3to2}). If a homoclinic forms from an H-chain that involves more than one zone, two homoclinic separatrices from each zone are required to break since in total: the start homoclinic separatrix, the end homoclinic separatrix, and all the homoclinic separatrices in between that join the consecutive zones (one to enter the zone, one to leave the zone) must break (see Figures \ref{Hchain1} and \ref{simultaneoushom}).  These observations lead us to define \emph{admissible} paths in the H-graph (see Figure \ref{Hgraphpath1}).  \par
%----------------------------
\begin{definition}
\label{admissiblepath}
  An \emph{admissible path} is a path in the H-graph such that it respects the orientation of the edges of the H-graph and uses only one edge per connected component of the disk minus the homoclinic separatrices.
\end{definition}
E.g. the path in Figure \ref{Hgraph} from $s_{5,6} \rightarrow s_{7,0} \rightarrow s_{1,2}$ would not be an admissible path since it uses two edges in a single connected component.
 %%%%%%%%%%%%%%%%%%%%
\begin{figure}[htbp]%
\large
    %\frame{
    \begin{pspicture}(0,0)(12,7)%
        \put(0.3,0){\includegraphics[width=.95\textwidth]{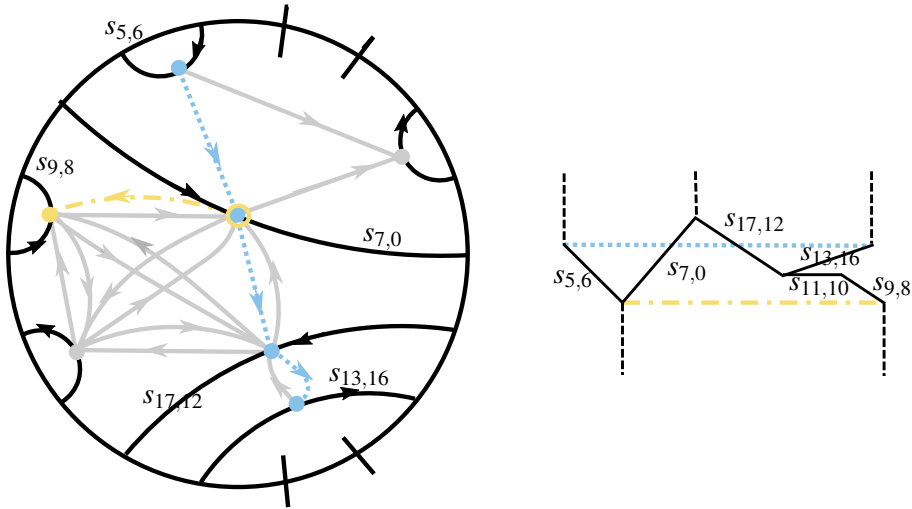}}%
        \put(5,3.55){\color[rgb]{0,0,0}\makebox(0,0)[lb]{\smash{$s_{7,0}$}}}%
        %\put(4.5,4.5){\color[rgb]{0,0,0}\makebox(0,0)[lb]{\smash{$s_{1,2}$}}}%
        \put(1.6,6.3){\color[rgb]{0,0,0}\makebox(0,0)[lb]{\smash{$s_{5,6}$}}}%
        \put(0.7,4.5){\color[rgb]{0,0,0}\makebox(0,0)[lb]{\smash{$s_{9,8}$}}}%
        %\put(1.3,2){\color[rgb]{0,0,0}\makebox(0,0)[lb]{\smash{$s_{11,10}$}}}%
        \put(2.1,1.4){\color[rgb]{0,0,0}\makebox(0,0)[lb]{\smash{$s_{17,12}$}}}%
        \put(4.55,1.7){\color[rgb]{0,0,0}\makebox(0,0)[lb]{\smash{$s_{13,16}$}}}%
        \put(7.5,3){\color[rgb]{0,0,0}\makebox(0,0)[lb]{\smash{$s_{5,6}$}}}%
        \put(9,3.1){\color[rgb]{0,0,0}\makebox(0,0)[lb]{\smash{$s_{7,0}$}}}%
        \put(9.75,3.75){\color[rgb]{0,0,0}\makebox(0,0)[lb]{\smash{$s_{17,12}$}}}%
        \put(10.75,3.3){\color[rgb]{0,0,0}\makebox(0,0)[lb]{\smash{$s_{13,16}$}}}%
        \put(10.6,2.9){\color[rgb]{0,0,0}\makebox(0,0)[lb]{\smash{$s_{11,10}$}}}%
        \put(11.7,2.9){\color[rgb]{0,0,0}\makebox(0,0)[lb]{\smash{$s_{9,8}$}}}%
    \end{pspicture}
    %}
\caption{Two \emph{admissible} paths (see Definition \ref{admissiblepath}) in the H-graph (left) and the corresponding homoclinic separatrices in rectifying coordinates (right). The union of these two paths is also a valid union graph, since the two paths neither begin or end at the same vertex, and the direction through each vertex is well-defined (see Definition \ref{uniongraph}). The vertical dashed lines in rectifying coordinates are not (necessarily) to be seen as identified as in Figures \ref{sumdynres}, \ref{defex0}, and \ref{simplehomjoin}. They represent  a piece of an unspecified zone that has those homoclinic separatrices on its lower or upper boundary. The same remark holds for the vertical dashed lines in Figures \ref{product},  \ref{Hchain1}, \ref{simultaneoushom}, \ref{5to4}, and \ref{3to2}. }
\label{Hgraphpath1}
\end{figure}
%----------------------------
\paragraph{Simultaneous Paths} The H-graph was designed so that paths correspond to single homoclinic separatrices that can appear under  perturbation.
To consider the possibility when more than one homoclinic may form at once, \emph{simultaneous} paths in the H-graph need to be considered (see Figure \ref{Hgraphpath1}). Union path graphs are defined as allowed unions of admissible paths.
\begin{definition}
\label{uniongraph}
A \emph{union graph} is a union  of admissible paths such that:
\begin{itemize}
  \item
    No pair of paths can share a start or end vertex, i.e. there can not be a vertex which is a start vertex for two paths or an end vertex for two paths (see Figure \ref{conflict}).
  \item
      Each vertex in the union of these admissible paths has a well-defined direction through it, explained further in the following. Consider the vertex in question of the union of admissible paths. This vertex corresponds to a homoclinic separatrix in the initial configuration, and there are two connected components (zones) in the disk that have this homoclinic on the boundary.  By "well-defined direction" through the vertex, we mean that in one connected component, all edges of the union of paths must be incoming to the vertex, and in the other connected component, all edges of the union of paths must be outgoing from that vertex (see Figure \ref{directionflow} for a violation of this criterion).
      \end{itemize}
\end{definition}
We explain the need for the two conditions in the definition. If two paths shared a start (resp. end) vertex, it would be the same as saying $s_{k,j_1}$ and $s_{k,j_2}$ (resp. $s_{k_1,j}$ and $s_{k_2,j}$) could form simultaneously, but it is not possible for two homoclinic separatrices to share an asymptotic direction. If the vertices do not have a well-defined direction through them, it corresponds to a single $\widetilde{\tau}$ taking on a value in $\mathbb{H}_+$ and $\mathbb{H}_-$ simultaneously, which is not possible (see Figures \ref{directionflow} and \ref{conflict}).
\begin{remark}
  We do not claim that every union graph defined as above corresponds to a possible simultaneous formation of homoclinic separatrices, only that a simultaneous formation of homoclinics must be of that form.
\end{remark}
%----------------------------
%%%%%%%%%%%%%%%%%%%%%%%%%%%%%%%
%----------------------------
\begin{figure}[htbp]%
\large
    %\frame{
    \begin{pspicture}(0,0)(12,3.5)%
        \put(0.3,0){\includegraphics[width=.45\textwidth]{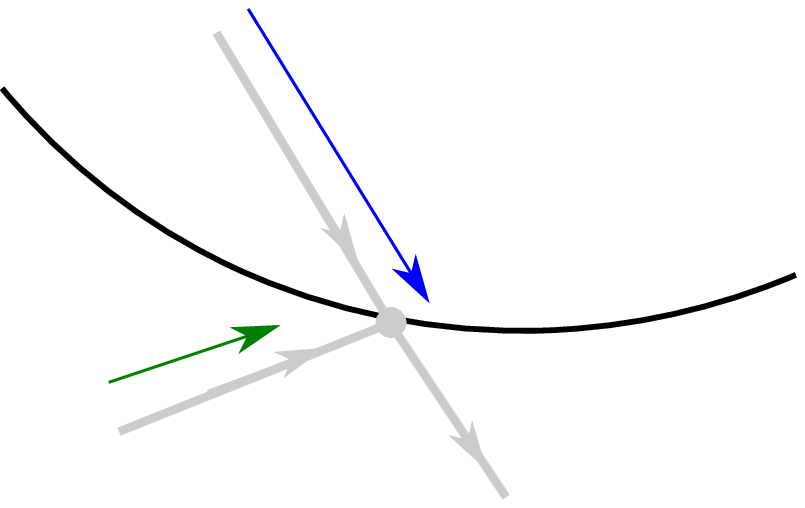}}%
        \put(6.3,0){\includegraphics[width=.45\textwidth]{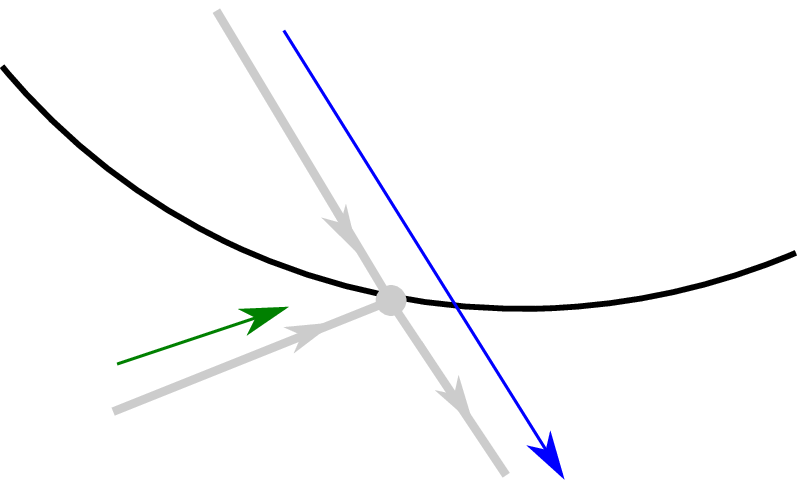}}%
        \put(3,2.5){\color[rgb]{0,0,0}\makebox(0,0)[lb]{\smash{$p_1$}}}%
        \put(1.3,1.3){\color[rgb]{0,0,0}\makebox(0,0)[lb]{\smash{$p_2$}}}%
        \put(9,2.4){\color[rgb]{0,0,0}\makebox(0,0)[lb]{\smash{$p_1$}}}%
        \put(7.3,1.3){\color[rgb]{0,0,0}\makebox(0,0)[lb]{\smash{$p_2$}}}%
    \end{pspicture}
    %}
\caption{A union graph (see Definition \ref{uniongraph}) is a union  of admissible  paths that satisfies the following two criteria.
    1. There can not be a vertex which is a start vertex for two paths or an end vertex for two paths (as is pictured on the left). 2. Each vertex in the union of these admissible paths has a well-defined direction through it (see Definition \ref{uniongraph}). An example of a violation of this criterion is pictured on the right and in Figure \ref{directionflow}.}
\label{conflict}
\end{figure}
%%%%%%%%%%%%%%%%%%%%%%%%%%%%%%
\begin{figure}[htbp]%
\large
    %\frame{
    \begin{pspicture}(0,0)(12,4)%
        \put(0,0){\includegraphics[width=.75\textwidth]{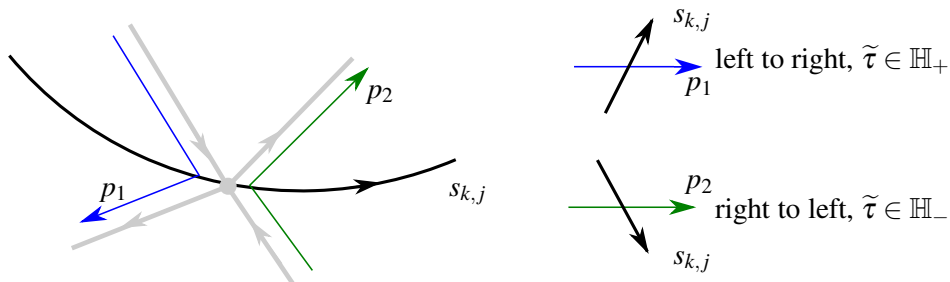}}%
        \put(1.2,1.2){\color[rgb]{0,0,0}\makebox(0,0)[lb]{\smash{$p_1$}}}%
    \put(4.7,2.5){\color[rgb]{0,0,0}\makebox(0,0)[lb]{\smash{$p_2$}}}%
    \put(8.75,0.3){\color[rgb]{0,0,0}\makebox(0,0)[lb]{\smash{$s_{k,j}$}}}%
    \put(8.9,2.6){\color[rgb]{0,0,0}\makebox(0,0)[lb]{\smash{$p_1$}}}%
    \put(5.8,1.2){\color[rgb]{0,0,0}\makebox(0,0)[lb]{\smash{$s_{k,j}$}}}%
    \put(9.3,2.9){\color[rgb]{0,0,0}\makebox(0,0)[lb]{\smash{left to right, $\widetilde{\tau} \in \mathbb{H}_+$}}}%
    \put(8.9,1.3){\color[rgb]{0,0,0}\makebox(0,0)[lb]{\smash{$p_2$}}}%
    \put(8.75,3.5){\color[rgb]{0,0,0}\makebox(0,0)[lb]{\smash{$s_{k,j}$}}}%
    \put(9.3,0.9){\color[rgb]{0,0,0}\makebox(0,0)[lb]{\smash{right to left, $\widetilde{\tau}  \in \mathbb{H}_-$}}}%
    \end{pspicture}
    %}
\caption{ If the vertices of the union graph do not have a well-defined direction through them, it corresponds to a single $\widetilde{\tau}$ taking on a value in $\mathbb{H}_+$ and $\mathbb{H}_-$ simultaneously, which is not possible. In the example above,  consider two paths $p_1$ and $p_2$ which have opposite orientation through the vertex corresponding to $s_{k,j}$. If $p_1$ is a path, then necessarily the $\widetilde{\tau} _{k,j} \in \mathbb{H}_+$. To see this, note that $p_1$ crosses from the left of $s_{k,j}$ to the right of $s_{k,j}$. Looking at this in rectifying coordinates (right), one can see that $\widetilde{\tau}_{k,j} \in \mathbb{H}_+$. Similarly, if $p_2$ were a path, $\widetilde{\tau} _{k,j} \in \mathbb{H}_-$, so clearly both $p_1$ and $p_2$ can not occur simultaneously.  }
\label{directionflow}
\end{figure}
%%%%%%%%%%%%%%%%%%%%%%%%%%%%%%%
\begin{proposition}
\label{nocyclesprop}
The union graph  contains neither directed nor undirected cycles.
\end{proposition}
\begin{proof}
Suppose the union graph does contain a cycle.  Pick a vertex $v$ in the cycle and let $e_1$ and $e_2$ be the edges in the cycle that meet at $v$. \par
%%%%%-------------------------
\begin{figure}[htbp]%
\large
    %\frame{
    \begin{pspicture}(0,0)(12,3.5)%
        \put(1.3,0){\includegraphics[width=.8\textwidth]{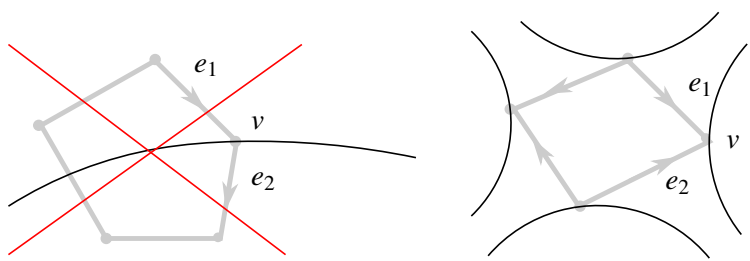}}%
    \put(4.5,1.75){\color[rgb]{0,0,0}\makebox(0,0)[lb]{\smash{$v$}}}%
    \put(10.75,1.5){\color[rgb]{0,0,0}\makebox(0,0)[lb]{\smash{$v$}}}%
    \put(3.75,2.5){\color[rgb]{0,0,0}\makebox(0,0)[lb]{\smash{$e_1$}}}%
    \put(4.5,1){\color[rgb]{0,0,0}\makebox(0,0)[lb]{\smash{$e_2$}}}%
    \put(10.25,2.25){\color[rgb]{0,0,0}\makebox(0,0)[lb]{\smash{$e_1$}}}%
    \put(9.95,0.95){\color[rgb]{0,0,0}\makebox(0,0)[lb]{\smash{$e_2$}}}%
    \end{pspicture}
   % }
   \caption{(Left) There can not be a cycle in the union graph that contains a vertex $v$ such that $e_1$ is incoming to $v$ and $e_2$ is outgoing from $v$. This would imply that $e_1$ and $e_2$ are on opposite sides of the homoclinic that goes through $v$, otherwise this would violate direction flow.  Since each homoclinic must traverse the entire disk, it must cut the cycle somewhere else besides at $v$.  We arrive at a contradiction since  a homoclinic can not cross two vertices, nor can it cross an edge in the H-graph. (Right) If the union graph had a cycle, it would have to have an even number of edges with alternating orientation and would have to be entirely contained in a single connected component.  }
   \label{CaseI}
\end{figure}
%%----------------------------
  \paragraph{Case I:} The cycle contains a vertex $v$ such that $e_1$ is incoming to $v$ and $e_2$ is outgoing from $v$ (see left of Figure \ref{CaseI}).  Now $e_1$ and $e_2$ must be on opposite sides of the homoclinic that goes through $v$, otherwise this would violate direction flow (the second condition in Definition \ref{uniongraph}).  Since each homoclinic must traverse the entire disk, it must cut the cycle somewhere else besides at $v$.  We arrive at a contradiction since  a homoclinic can not cross two vertices, nor can it cross an edge in the H-graph (see left of Figure \ref{CaseI}). Therefore, there can be no cycles in the union graph with such a vertex. This eliminates the possibility of a cycle having an odd number of edges or exactly two edges, which would necessarily have such a vertex. \par
  \paragraph{Case II:} For all vertices $v$ in the cycle, either both $e_1$ and $e_2$ are incoming to $v$ or both are outgoing from $v$ (since not in Case I). This implies there are an even number of edges in the cycle, and the orientation of the edges is alternating (see right of Figure \ref{CaseI}). Furthermore, any such cycle would have to be entirely contained in a single connected component of the disk minus the initial homoclinic configuration. Otherwise, a homoclinic would cross its interior, and this is impossible by the same argument as in Case I. The orientation of all homoclinics on the boundary of the connected component is the same (connected component to the left of all homoclinics, or the connected component is to the right of all homoclinics on the boundary). In fact, this connected component must be a cylinder zone, since the cycle has alternating orientation of edges and therefore must contain edges contrary to the orientation of the H-chain on the boundary (see Remark \ref{directionofedges}). We will use two edges in the cycle which have 'opposite' orientation to the orientation of the H-chain on the boundary to arrive at a contradiction. There must exist at least two such edges in every cycle with an even number of edges greater than 2 or we would be in Case I (see Figure \ref{nocycles}).
    %%%%%%%%%%%%%%%%%%%%%%%%%%%
  \begin{figure}[htbp]%
\large
    %\frame{
    \begin{pspicture}(0,0)(12,6.2)%
        \put(1.5,0){\includegraphics[width=.7\textwidth]{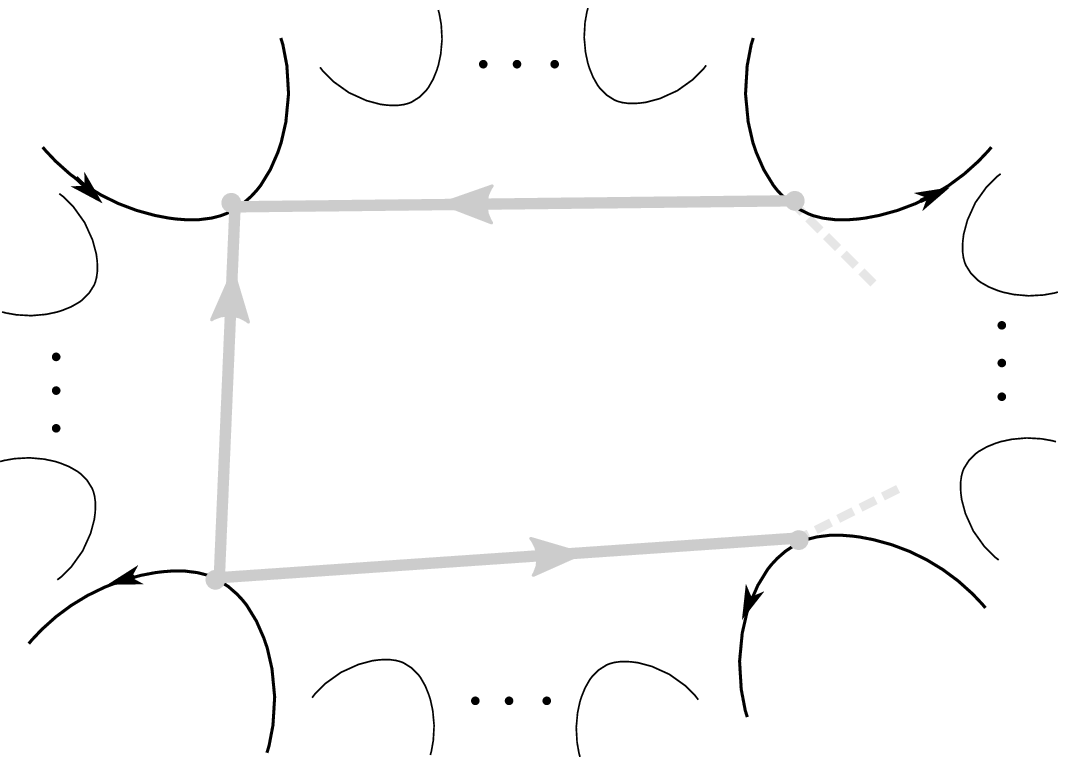}}%
        \put(9.5,2.2){\color[rgb]{0,0,0}\makebox(0,0)[lb]{\smash{$s_N$}}}%
    \put(8,1.2){\color[rgb]{0,0,0}\makebox(0,0)[lb]{\smash{$s_1$}}}%
    \put(5,1){\color[rgb]{0,0,0}\makebox(0,0)[lb]{\smash{\text{$m$ between}}}}%
    \put(2.1,3){\color[rgb]{0,0,0}\makebox(0,0)[lb]{\smash{\text{$n$ between}}}}%
    \put(5,5){\color[rgb]{0,0,0}\makebox(0,0)[lb]{\smash{\text{$\ell$ between}}}}%
    \put(2.5,1){\color[rgb]{0,0,0}\makebox(0,0)[lb]{\smash{$s_{m+1}$}}}%
    \put(2.5,5){\color[rgb]{0,0,0}\makebox(0,0)[lb]{\smash{$s_{m+n+1}$}}}%
    \put(8,5){\color[rgb]{0,0,0}\makebox(0,0)[lb]{\smash{$s_{m+n+\ell+1}$}}}%
  \end{pspicture}
  %}
   \caption{If the union graph contained a cycle, then it would be entirely contained in a single connected component of the disk model and the orientation of the edges would be alternating. Any such cycle contains a segment of three edges in the H-graph, where the first and third are contrary to the orientation of the H-chain on the boundary and the second has the same orientation as the H-chain on the boundary. }
   \label{nocycles}
\end{figure}
%%%%%%%%%%%%%%%%%%%%%%%%%%%%%%%
  In particular, any such cycle contains a segment of three edges in the H-graph, where the first and third are contrary to the orientation of the H-chain on the boundary and the second has the same orientation as the H-chain on the boundary (see Figure \ref{nocycles}). We will focus on the two contrary edges in such a segment. Each of those two edges corresponds to two different new homoclinics that form, since no single path in the H-graph  can contain two edges from the same connected component (see Definition \ref{admissiblepath}).
  Pick one edge that has opposite orientation to the H-chain (e.g. $s_{m+1} \rightarrow s_1$ in Figures \ref{nocycles} and \ref{nocycles2}), and cut the corresponding cylinder in rectifying coordinates such that the first  homoclinic in the H-chain is the starting  homoclinic (vertex) for that edge in the H-graph (see top of Figure \ref{nocycles2} where the cut is made at $s_{m+1}$).
%%----------------------------
  \begin{figure}[htbp]%
\large
    %\frame{
    \begin{pspicture}(0,0)(12,6)%
        \put(1.3,0){\includegraphics[width=.8\textwidth]{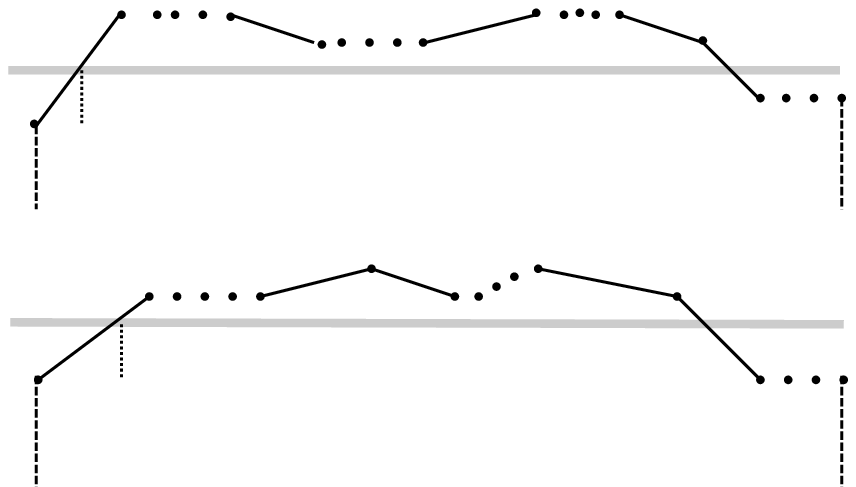}}%
        \put(2.3,4.4){\color[rgb]{0,0,0}\makebox(0,0)[lb]{\smash{$b$}}}%
        \put(2.7,1.4){\color[rgb]{0,0,0}\makebox(0,0)[lb]{\smash{$c$}}}%
        \put(1.6,5.3){\color[rgb]{0,0,0}\makebox(0,0)[lb]{\smash{$s_{m+1}$}}}%
    \put(4.3,5.4){\color[rgb]{0,0,0}\makebox(0,0)[lb]{\smash{$s_{m+n+1}$}}}%
    \put(6.2,5.6){\color[rgb]{0,0,0}\makebox(0,0)[lb]{\smash{$s_{m+n+\ell+1}$}}}%
    \put(8.7,5.5){\color[rgb]{0,0,0}\makebox(0,0)[lb]{\smash{$s_N$}}}%
    \put(9.55,5.05){\color[rgb]{0,0,0}\makebox(0,0)[lb]{\smash{$s_1$}}}%
    \put(0.6,1.6){\color[rgb]{0,0,0}\makebox(0,0)[lb]{\smash{$s_{m+n+\ell+1}$}}}%
    \put(4.5,2.45){\color[rgb]{0,0,0}\makebox(0,0)[lb]{\smash{$s_N$}}}%
    \put(5.8,2.45){\color[rgb]{0,0,0}\makebox(0,0)[lb]{\smash{$s_1$}}}%
    \put(8,2.45){\color[rgb]{0,0,0}\makebox(0,0)[lb]{\smash{$s_{m+n}$}}}%
    \put(9.9,1.5){\color[rgb]{0,0,0}\makebox(0,0)[lb]{\smash{$s_{m+n+1}$}}}%
 \end{pspicture}
 %}
   \caption{If the union graph contained a cycle, the cycle must be contained in a single connected component which is a cylinder zone, and the cycle has edges which alternate in orientation (see Figure \ref{nocycles} which depicts a clockwise cylinder). Any such cycle contains a segment of three edges in the H-graph, where the first and third are contrary to the orientation of the H-chain on the boundary and the second has the same orientation as the H-chain on the boundary (see Figure \ref{nocycles}). For each of the two edges that have opposite orientation to the H-chain, cut the  cylinder in rectifying coordinates such that the first  homoclinic in the H-chain is the starting  homoclinic for that edge in the H-graph.
   Above is shown two different cuts in rectifying coordinates of the same cylinder, seen as lower half-infinite strips with identified dashed vertical lines. The cuts are made at $s_{m+1}$ (top) and $s_{m+n+\ell+1}$ (bottom), corresponding to the start vertices of the contrary directed edges $s_{m+1} \rightarrow s_1$ and $s_{m+n+\ell+1} \rightarrow s_{m+n+1}$ in Figure \ref{nocycles}.
     In order for the new homoclinic to pass through the first and last homoclinic separatrices corresponding to these directed edges in the union graph, the ordered partial sums from left to right of the analytic invariants for the H-chain need to satisfy a set of inequalities.  These two sets of inequalities lead to a contradiction. }
   \label{nocycles2}
\end{figure}
%%----------------------------
  If this edge ($s_{m+1} \rightarrow s_1$) is in the union graph, it corresponds to a new homoclinic (path in H-graph) which enters the cylinder at $s_{m+1}$ and leaves the cylinder at $s_1$ (possibly beginning at $s_{m+1}$ and/or stopping at $s_1$). In order for the new homoclinic (path in H-graph which contains the edge $s_{m+1} \rightarrow s_1$) to pass through $s_{m+1}$ and $s_1$ without leaving the cylinder in between, the ordered partial sums from left to right of the analytic invariants for the H-chain need to all be greater than $b \geq 0$, the height of the new homoclinic relative to the first vertex in the cut cylinder ($b$ would be determined by the sums of the other analytic invariants in earlier parts of the chain), and the final partial sum needs to be less than or equal to $b$ (see the top of Figure \ref{nocycles2}). The same condition must apply to the other  opposite oriented edge in the cycle segment.  This will lead to a contradiction. Let $s_{m+1}$ and $s_1$ be the corresponding start and end homoclinics for an opposite oriented edge in the H-graph (see Figure \ref{nocycles}).
    Let $a_{m+1},\dots, a_1$ be the imaginary parts of the corresponding analytic invariants after perturbation. For the directed edge $s_{m+1} \rightarrow s_1$ to be realized, we need
  \begin{align}\label{edge1}
    a_{m+1} & >b \nonumber \\
    a_{m+1}+a_{m+2} &  >b \nonumber \\
    \dots & \nonumber \\
    a_{m+1}+a_{m+2}+\dots +a_{m+n+1} &  >b \nonumber \\
    \dots & \nonumber \\
    a_{m+1}+a_{m+2}+\dots+a_{m+n+1}+\dots +a_{m+n+\ell+1} &  >b \nonumber \\
    \dots & \nonumber \\
    a_{m+1}+a_{m+2}+\dots+a_{m+n+1}+\dots +a_{m+n+\ell+1}+ \dots +a_N &  >b \nonumber \\
    a_{m+1}+a_{m+2}+\dots+a_{m+n+1}+\dots +a_{m+n+\ell+1}+ \dots +a_N +a_1&  \leq b. \nonumber
  \end{align}
  Now consider the other opposite oriented edge in the cycle segment with starting and ending homoclinics $s_{m+n+\ell+1}$ and $s_{m+n+1}$ respectively (see Figure \ref{nocycles}). Let $c \geq 0$ be the height of the new homoclinic relative to the first vertex in this cut of the cylinder. Than for directed edge $s_{m+n+\ell+1}\rightarrow s_{m+n+1}$ to be realized, we need
   \begin{align}
    a_{m+n+\ell+1} & >c \nonumber \\
    \dots & \nonumber \\
    a_{m+n+\ell+1}+\dots +a_N+a_1 & >c \nonumber \\
    \dots & \nonumber\\
    a_{m+n+\ell+1}+\dots +a_N+a_1+ \dots +a_{m+n} & >c \nonumber \\
    a_{m+n+\ell+1}+\dots +a_N+a_1+ \dots +a_{m+n}+a_{m+n+1} & \leq c.\nonumber
  \end{align}
  By the first set of inequalities that $a_{m+1} + \cdots + a_{m+n+\ell}>b$, and by the second set of inequalities, $a_{m+n+\ell+1}+\cdots +a_{N}+a_1>c$. Then the last inequality in the first set  becomes $(>b) + (>c) \leq b$, a contradiction.%\qed
\end{proof}
%%----------------------------
\begin{proof}[of Theorem \ref{rank1characterization}]
  Assume throughout the rest of the proof that only rank 1,  multiplicity-preserving bifurcations are considered.\par
In order for a rank 1 bifurcation to occur, there must be  $n$ admissible (simultaneous) \emph{paths} that altogether use $n+1$ vertices.
The union graph for rank 1 bifurcations must be connected, since any two (or more) disjoint paths would require at least 4 vertices, corresponding to rank 2 or more. The union graph is therefore a tree since it has no cycles by Proposition \ref{nocyclesprop}. The union graph for rank 1 bifurcations must actually be a path itself (only two leaves)  by the following (see Figure \ref{rank1union}). Each homoclinic separatrix has two asymptotic directions: one outgoing from infinity (odd), and one incoming to infinity (even). It follows that if $n+1$ homoclinics break, and $n$ are to form,  only one odd and one even asymptotic direction can be lost in total. There can only be two leaves since each leaf of the union tree corresponds to a lost asymptotic direction. Furthermore, all the vertices besides the two leaves in the union of paths must be both the end of a path and the beginning of a (different) path.
%%----------------------------
  \begin{figure}[htbp]%
\large
    %\frame{
    \begin{pspicture}(0,0)(12,5.5)%
        \put(0.8,0){\includegraphics[width=.85\textwidth]{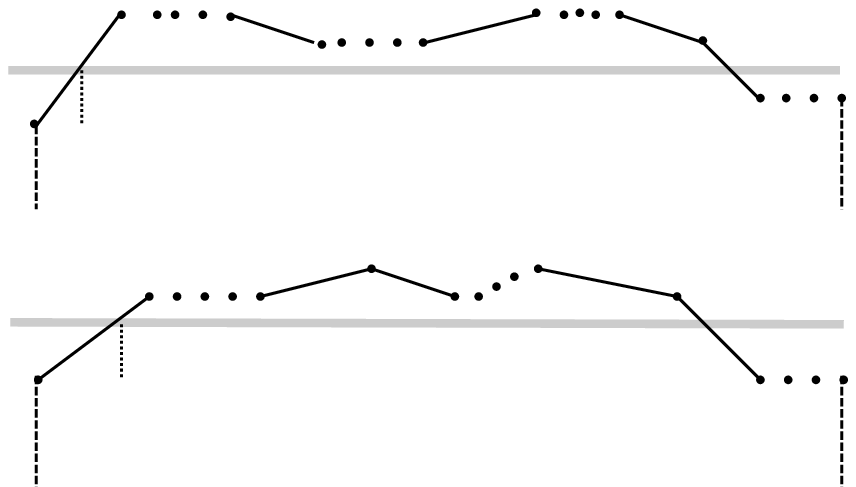}}%
       % \put(2.5,4.5){\color[rgb]{0,0,0}\makebox(0,0)[lb]{\smash{$b$}}}%
%        \put(2.7,0.2){\color[rgb]{0,0,0}\makebox(0,0)[lb]{\smash{$c$}}}%
%        \put(1.6,5.3){\color[rgb]{0,0,0}\makebox(0,0)[lb]{\smash{$s_{m+1}$}}}%
%    \put(4.3,5.4){\color[rgb]{0,0,0}\makebox(0,0)[lb]{\smash{$s_{m+n+1}$}}}%
%    \put(6.2,5.6){\color[rgb]{0,0,0}\makebox(0,0)[lb]{\smash{$s_{m+n+\ell+1}$}}}%
%    \put(8.7,5.5){\color[rgb]{0,0,0}\makebox(0,0)[lb]{\smash{$s_N$}}}%
%    \put(9.55,5.05){\color[rgb]{0,0,0}\makebox(0,0)[lb]{\smash{$s_1$}}}%
%    \put(0.6,1.6){\color[rgb]{0,0,0}\makebox(0,0)[lb]{\smash{$s_{m+n+\ell+1}$}}}%
%    \put(4.5,2.45){\color[rgb]{0,0,0}\makebox(0,0)[lb]{\smash{$s_N$}}}%
%    \put(5.8,2.45){\color[rgb]{0,0,0}\makebox(0,0)[lb]{\smash{$s_1$}}}%
%    \put(8,2.45){\color[rgb]{0,0,0}\makebox(0,0)[lb]{\smash{$s_{m+n}$}}}%
%    \put(9.9,1.5){\color[rgb]{0,0,0}\makebox(0,0)[lb]{\smash{$s_{m+n+1}$}}}%
 \end{pspicture}
 %}
   \caption{(Left) The union graph for rank 1, multiplicity-preserving bifurcations must itself be a directed path, consisting of $n+1$ vertices and $n$ paths of length 1. (Right) Such a union graph corresponds to the bifurcations as in the statement of Theorem \ref{rank1characterization}.}
   \label{rank1union}
\end{figure}
%%%---------------------------
Since each vertex has a well-defined direction through it, the union path has an orientation. Choose the leaf that is the tail of a directed edge -- it is the beginning of a path. This path can only consist of one edge, since the second vertex must have a path end there and it can not come from the other direction because of the orientation of the union path. Unless the second vertex is also the last vertex of the union path, a second path must start at the second vertex, and end at the third vertex for the same reason.  The same argument must hold for subsequent vertices, so all paths must be of length 1. This corresponds exactly to the type of bifurcation as described in the statement of the theorem.
\end{proof}
\begin{acknowledgements}
There are many people to thank for patiently listening to various stages of this work and providing helpful suggestions. In particular, the author is indebted to Bodil Branner, Frederick Gardiner and the other participants in the Extremal Length seminar at CUNY Graduate Center, Christian Henriksen, Poul G. Hjorth, Louis Pedersen, and Tan Lei.  The author would also like to express gratitude to the anonymous referee.
\end{acknowledgements}
%% Authors must disclose all relationships or interests that
%% could have direct or potential influence or impart bias on
%% the work:
%%
 \section*{Conflict of interest}
 The authors declare that they have no conflict of interest.
 %%%%%%%%%%%%%%%%%%%%%%%%%
% BibTeX users please use one of
%\bibliographystyle{spbasic}      % basic style, author-year citations
\bibliographystyle{spmpsci}      % mathematics and physical sciences
\bibliography{MultPresBifs}   % name your BibTeX data base
%%%%%%%%%%%%%%
\end{document}